\documentclass[12pt,a4paper,fleqn]{article}
\usepackage{latexsym}
\usepackage{amsmath}
\usepackage{amsthm}
\usepackage{amssymb}
\usepackage{amsfonts}
\usepackage{enumerate}
\usepackage{graphicx}
\usepackage[usenames,dvipsnames]{xcolor}
\usepackage{rotating}
\usepackage{longtable}
\usepackage{caption}
\newtheorem{theorem}{Theorem}[section]
\newtheorem*{theorem*}{Theorem}
\newtheorem{lemma}[theorem]{Lemma}
\newtheorem*{lemma*}{Lemma}
\newtheorem{corollary}[theorem]{Corollary}
\newtheorem*{corollary*}{Corollary}
\newtheorem{proposition}[theorem]{Proposition}
\newtheorem*{proposition*}{Proposition}

{\theoremstyle{definition}
\newtheorem{definition}[theorem]{Definition}
\newtheorem*{definition*}{Definition}
\newtheorem{example}[theorem]{Example}
\newtheorem*{example*}{Example}
\newtheorem*{examples*}{Examples}
\newtheorem{remark}[theorem]{Remark}
\newtheorem*{remark*}{Remark}
\newtheorem{result}{Result}

\newtheorem*{question*}{Question}

\newtheorem*{problem*}{Problem}

\newtheorem*{notation*}{Notation}

\newtheorem*{note*}{Note}

\newtheorem*{algorithm*}{Algorithm}
}
\renewcommand{\arraystretch}{0.75}

\newcommand{\Colon}{\:\mathbin{:}\:}

\newcommand{\shape}{\mathbin{\mbox{\rm sh}}\,}

\newcommand{\proofblock}{\hspace{\fill}\fbox{\rule{0pt}{3pt}\rule{3pt}{0pt}}}
\newcommand{\LEQ}{\leqslant}
\newcommand{\GEQ}{\geqslant}
\let\lhd\triangleleft
\let\unlhd\trianglelefteq

\newcounter{xcoord}
\newcounter{ycoord}

\newcommand{\nodeB}[3]{%
\setcounter{xcoord}{#1}
\addtocounter{xcoord}{#3}
\addtocounter{xcoord}{-2}
\setcounter{ycoord}{#2}
\addtocounter{ycoord}{-#3}
\addtocounter{ycoord}{2}
\put(\thexcoord,\theycoord){\circle*{1}}
}

\newcommand{\cl}[1]{%
\ifnum#1=1 \color{violet}
\else\ifnum#1=2 \color{JungleGreen}
\else\ifnum#1=3 \color{red}
\else\ifnum#1=4 \color{blue}
\else\ifnum#1=5 \color{orange}
\else\ifnum#1=6 \color{cyan}
\else\ifnum#1=7 \color{Maroon}
\fi\fi\fi\fi\fi\fi\fi%
}
\newcommand{\clx}[1]{%
\ifnum#1=1 \color{violet} 1
\else\ifnum#1=2 \color{JungleGreen} 2
\else\ifnum#1=3 \color{red} 3
\else\ifnum#1=4 \color{blue} 4
\else\ifnum#1=5 \color{orange} 5
\else\ifnum#1=6 \color{cyan} 6
\else\ifnum#1=7 \color{Maroon} 7
\fi\fi\fi\fi\fi\fi\fi%
}

\title{On ordered $k$-paths and  rims for certain families of Kazhdan-Lusztig cells of $S_n$}
\author{T. P. McDonough%
\thanks{%
\textit{Department of Mathematics,
Aberystwyth University,
Aberystwyth SY23 3BZ,
United Kingdom.}
{E-mail: tpd@aber.ac.uk}
}
\ and\ \
C. A. Pallikaros%
\thanks{%
\textit{Department of Mathematics and Statistics,
University of Cyprus,
P.O.Box 20537,
1678 Nicosia,
Cyprus.}
{E-mail: pallikar@ucy.ac.cy}
}
}
\date{3rd September, 2019}
\begin{document}
\maketitle

\begin{abstract}
For a composition $\lambda$ of $n$ we consider the Kazhdan-Lusztig cell in the symmetric group $S_n$ containing the longest element of the standard parabolic subgroup of $S_n$  associated to $\lambda$.
In this paper we extend some  of the ideas and results in [{Beitr{\"a}ge} zur Algebra und Geometrie, \textbf{59} (2018), no.~3, 523--547].
In particular, by introducing the notion of an ordered $k$-path, we are able to obtain alternative explicit descriptions for some additional families of cells associated to compositions.
This is achieved by first determining the rim of the cell, from which reduced forms for all the elements of the cell are easily obtained.
\\[1.5ex]
\end{abstract}

\section{Introduction}

In~\cite{KLu79} Kazhdan and Lusztig introduced, via certain preorders, the left cells, the right cells and the two-sided cells of a Coxeter group as a means of investigating the representation theory of the Coxeter group and its associated Hecke algebra.

In the case of the symmetric group $S_n$, the cell to which an element belongs can be determined by an application of the Robinson-Schensted process.
Moreover, one can obtain all the elements in a given cell by applying the  reverse Robinson-Schensted process but, unfortunately, this does not lead to some straightforward way of obtaining reduced forms for these elements.
A useful observation is that each right (resp., left) cell of $S_n$ contains a unique involution.

The present paper, which is a continuation of the work in~\cite{MPa08,MPa15,MPa17}, is concerned with the problem of determining reduced expressions for all the elements in a given cell.
As in~\cite{MPa15,MPa17}, we again focus attention on (right) cells which have the property that the unique involution they contain is the longest element in some standard parabolic subgroup of $S_n$ and which are thus associated to compositions $\lambda$ of $n$.
The motivation for some of the main ideas in~\cite{MPa17}, which we also use here, emanates from work in~\cite{Schensted1961} and~\cite{Greene1974} on increasing and decreasing subsequences.
By extending various ideas in~\cite{MPa17} we are able to obtain alternative explicit descriptions for
some additional families of Kazhdan-Lusztig cells via the determination of their rim.
This directly leads to determining reduced forms for all the elements in these cells.

The paper is organized as follows:
In Section~\ref{sec:2c} we recall some basic facts about Kazhdan-Lusztig cells in the symmetric group.
We also recall some useful properties of paths and admissible diagrams described in~\cite{MPa17}.
Moreover, in Theorem~\ref{thm:2.16a} we characterize all compositions of $n$ for which the rim of the associated cell consists of precisely one element.

In Section~\ref{sec:OrdPaths} we introduce the notion of an ordered $k$-path which plays a key role in proving some of the main results in this paper.
In Theorem~\ref{thm:2.14a} we show that every $k$-path in a diagram $D$ is equivalent to an ordered $k$-path.
At the end of Section~\ref{sec:OrdPaths} we establish some results on extending certain ordered $k$-paths in a particular way which turns out to be useful in various arguments that follow. 

Finally, in Section~\ref{sec:3a}, using the ideas and techniques developed earlier on in the paper, we prove the main results of the section (Theorems~\ref{thm:3.5a}, \ref{thm:3.7a} and~\ref{prop:3.15a}) on determining the rim and hence reduced forms for the elements in certain families of cells.
Another result which is very useful in this direction is Proposition~\ref{prop:3.2a}
which uses ideas in~\cite{MPa17} related to the induction of cells  (see~\cite{BVo83}) and which allows us to `transform' certain cells of $S_n$ into cells of $S_m$ with $m>n$.

\begin{section}
{Preliminaries and generalities}
\label{sec:2c}

\subsection{Kazhdan-Lusztig cells in the symmetric group}
For any Coxeter system $(W,S)$, Kazhdan and Lusztig \cite{KLu79}
introduced three preorders $\LEQ_L$, $\LEQ_R$ and $\LEQ_{LR}$, with
corresponding equivalence relations $\sim_L$, $\sim_R$ and $\sim_{LR}$,
whose equivalence classes are called
\emph{left cells}, \emph{right cells} and \emph{two-sided cells},
respectively.
Each cell of $W$ provides a representation of $W$, with the $C$-basis of the Hecke algebra $\mathcal H$ of $(W,S)$ playing an important role in the construction of this representation; see~\mbox{\cite[\S~1]{KLu79}}.

We refer  to \cite{GPf00} and \cite{Hum90} for basic concepts relating to Coxeter groups and Hecke algebras.
In particular, for a Coxeter system $(W,S)$,
$W_J=\langle J\rangle$ denotes the standard parabolic subgroup determined by a
subset $J$ of $S$, $w_J$ denotes the longest element of $W_J$
and $\mathfrak{X}_J$ denotes the set of minimum length elements in the
right cosets of $W_J$ in $W$
(the distinguished right coset representatives).
Also recall the \emph{prefix relation} on the elements of~$W$:
if $x,y\in W$ we say that $x$ is a \emph{prefix} of $y$ if $y$ has a \emph{reduced form} beginning with a reduced form for $x$.

The following result collects some useful propositions concerning cells.
For proofs of (i) and (ii), see \cite[2.3ac]{KLu79} and
\cite[5.26.1]{Lus84b} 
respectively.
\begin{result}[\cite{KLu79,Lus84b}]
\label{res:2c}
\mbox{}\par
\begin{tabular}[b]{rp{5.25in}}
(i) &
If $x,y,z$ are elements of $W$ such that $x$ is a prefix of $y$,
$y$ is a prefix of $z$ and $x\sim_R z$ then $x\sim_R y$.
\\
(ii) &
If $J\subseteq S$, then the right cell containing $w_J$ is contained
in $w_J\mathfrak{X}_J$.
\end{tabular}
\end{result}

In this paper we  focus on the symmetric group.
For the basic definitions and background concerning partitions, compositions, Young diagrams, Young tableaux and the  Robinson-Schensted
correspondence we refer to~\cite{Ful97} or~\cite{Sagan}.

The symmetric group $S_n$ (acting on the right) on $\{1,\dots,n\}$ is a Coxeter group with Coxeter system
$(W,S)$ where $W=S_n$, $S=\{s_1,\ldots,s_{n-1}\}$, and $s_i$ is the
transposition $(i,i+1)$.
An element $w$ of $W$ can be described in different forms:
as a word in the generators $s_1$, \ldots , $s_{n-1}$,
as products of disjoint cycles on $1,\ldots,n$,
and in \emph{row-form}
$[w_1,\ldots,w_n]$ where $w_i=iw$ for $i=1,\dots,n$.
The Coxeter \emph{length} $l(w)$ of the element $w\in W$, that is the
shortest length of a word in the elements of $S$ representing $w$,
has an easy combinatorial description;
$l(w)$ is the number of pairs $(w_i,w_j)$ with $i<j$ and $w_i>w_j$.
The longest element $w_0$ in $W$ is
the permutation defined by $i\mapsto n+1-i$.

Let $\Phi=\{\epsilon_i-\epsilon_j\colon 1\LEQ i,j\LEQ n,\ i\ne j\}$ and $\Phi^+=\{\epsilon_i-\epsilon_j\colon 1\LEQ i<j\LEQ n\}$ where $\{\epsilon_1,\ldots,\epsilon_n\}$ is an orthonormal basis of an $n$-dimensional Euclidean space; see~\cite[p.~41]{Hum90}.
There is an action of $S_n$  on $\Phi$ given by $(\epsilon_i-\epsilon_j)w=\epsilon_{iw}-\epsilon_{jw}$ ($w\in S_n$).
The Coxeter generator $s_i$ corresponds to the reflection in the hyperplane orthogonal to $\epsilon_i-\epsilon_{i+1}$.
For $w\in S_n$, we define $N^+(w)=\{\alpha\in\Phi^+\colon \alpha w\in\Phi^+\}$ and $N^-(w)=\Phi^+-N^+(w)$.
Then $l(w)=|N^-(w)|$; see~\cite[p.~14]{Hum90}.

All our partitions and compositions will be assumed to be \emph{proper}
(that is, with no zero parts).
We use the notation $\lambda\vDash n$ (respectively, $\lambda\vdash n$)
to say that $\lambda$ is a composition (respectively, partition) of
$n$.

Let $\lambda=(\lambda_1, \ldots , \lambda_r)$ be a \textit{composition}
of $n$ with $r$ \emph{parts}.
Recall that the \emph{conjugate} composition
$\lambda'=(\lambda_1',\ldots,\lambda_{r'}')$ of $\lambda$ is defined by
$\lambda'_i=\left|\{j\Colon 1\LEQ j\LEQ r\mbox{ and }
i\LEQ\lambda_j\}\right|$ for $1\LEQ i\LEQ r'$, where $r'$ be the maximum part of the composition $\lambda$.
It is immediate that $\lambda'$ is a partition of $n$ with $r'$ parts.
We also define the subset
$J(\lambda)$
\label{def:j_lam}
of $S$ to be
$S\backslash \{s_{\lambda_1},s_{\lambda_1+\lambda_2},\ldots,
s_{\lambda_1+\ldots+\lambda_{r-1}}\}$.
Thus, corresponding to the composition $\lambda$, there is a standard
parabolic subgroup of $W$, also known as a Young subgroup, whose Coxeter generator set is $J(\lambda)$.
The longest element $w_{J(\lambda)}$ of $W_{J(\lambda)}$ can be
described in row-form by concatenating the sequences
$(\widehat\lambda_{i+1}, \ldots ,\widehat\lambda_{i}+1)$
for $i=0,\dots,r-1$,
where $\widehat\lambda_0=0$, $\widehat\lambda_{r}=n$, and
$\widehat\lambda_{i+1}=\lambda_{i+1}+\widehat\lambda_{i}$.

If $\nu=(\nu_1, \ldots , \nu_r)\vdash n$ and $\mu=(\mu_1,\ldots,\mu_s)\vdash n$, write
$\nu\unlhd\mu$ if
$\sum_{1\LEQ i\LEQ k}\nu_i \LEQ \sum_{1\LEQ i\LEQ k}\mu_i$, for all $k$ with $1\LEQ k\LEQ s$.
This is the dominance order of partitions (see~\cite[p.~58]{Sagan}).
If $\nu\unlhd\mu$ and $\nu\neq \mu$, we write $\nu\lhd\mu$.

In the case of the symmetric group $S_n$, the Robinson-Schensted
correspondence gives a combinatorial method of identifying the
Kazhdan-Lusztig cells.
The Robinson-Schensted correspondence is a bijection of $S_n$
to the set of pairs of standard Young tableaux
$(\mathcal{P},\mathcal{Q})$ of the same shape and with $n$ entries,
where the shape of a tableau is the partition counting the number of
entries on each row.
Denote this correspondence by
$w\mapsto(\mathcal{P}(w),\mathcal{Q}(w))$.
Then $\mathcal{Q}(w)=\mathcal{P}(w^{-1})$.
The \emph{shape} of $w$, denoted by $\shape{w}$, is defined to be
the common shape of the Young tableaux $\mathcal{P}(w)$ and
$\mathcal{Q}(w)$.

The following result in~\cite{KLu79} (see also {\cite[Theorem A]{Ari00} or \cite[Corollary 5.6]{Gec05}}) characterises the cells in $S_n$:
If $\mathcal{P}$ is a fixed standard Young tableau then
the set $\{w\in W\Colon \mathcal{P}(w)=\mathcal{P}\}$
is a left cell of $W$ and the set
$\{w\in W\Colon\mathcal{Q}(w)=\mathcal{P}\}$ is a right cell of $S_n$.
Conversely, every left cell and every right cell arises in this way.
Moreover, the two-sided cells are the subsets of $W$ of the form
$\{w\in W\Colon\,\shape w \mbox{ is a fixed partition}\}$.

\subsection{Diagrams, rims and reduced forms}

We recall the generalizations of the notions of diagram and tableau,
commonly used in the basic theory, which we described in \cite{MPa15}.
A \emph{diagram} $D$ is a non-empty finite subset of $\mathbb{Z}^2$.
We will assume  that $D$ has no empty rows or columns.
These are the principal diagrams of \cite{MPa15}.
We will also assume that both rows and columns of $D$ are indexed
consecutively from 1; a node in $D$ will be given coordinates $(a,b)$ where $a$ and $b$ are the indices respectively of the row and column which the node belongs to
(rows are indexed from top to bottom and columns from left to right).
The \emph{row-composition} $\lambda_D$ (respectively,
\emph{column-composition} $\mu_D$) of $D$ is defined by
setting $\lambda_{D,k}$ (respectively, $\mu_{D,k}$) to be the number of nodes on
the $k$-th row (respectively, column) of $D$.
If $\lambda$ and $\mu$ are compositions,
we will write $\mathcal{D}^{(\lambda,\mu)}$ for the set
of (principal) diagrams $D$ with $\lambda_D=\lambda$ and $\mu_D=\mu$.
We also define $\mathcal{D}^{(\lambda)}=\bigcup_{\mu\vDash n}\mathcal{D}^{(\lambda,\mu)}$.
A well-known diagram associated with a partition
$\nu=(\nu_1,\ldots,\nu_r)$ is a \emph{Young diagram}
$V(\nu)=\{(i,j)\colon 1\LEQ i\LEQ r,\ 1\LEQ j\LEQ \nu_i\}$.
A \emph{special diagram} is a diagram obtained from a Young diagram
by permuting the rows and columns.
Special diagrams are characterised in the following proposition.

\begin{result}[{\cite[Proposition~3.1]{MPa15}}.\ %
{\rm Compare~\cite[Lemma~5.2]{DMP10}}]
\label{res:2a}
Let $D$ be a diagram. The following statements are equivalent.
(i)~$D$ is special;
(ii)~$\lambda_D''=\mu_D'$;
(iii)~for every pair of nodes $(i,j),(i',j')$ of $D$ with $i\neq i'$
and $j\neq j'$, at least one of $(i',j)$ and $(i,j')$ is also a node
of $D$.
\proofblock
\end{result}

Clearly if $\nu\vdash n$, then $V(\nu)$ is the unique element of $\mathcal D^{(\nu,\nu')}$.
It follows that $\mathcal D^{(\lambda,\mu)}$ consists of a single diagram, which is special, if $\lambda$ and $\mu$ are compositions of $n$ with $\lambda''=\mu'$.

If $D$ is a diagram of size $n$ (that is, consisting of precisely $n$ nodes),
a \emph{$D$-tableau} is a bijection
$t\Colon D\rightarrow \{1,\ldots,n\}$ and
we refer to $(i,j)t$, where $(i,j)\in D$, as the
$(i,j)$-\emph{entry} of~$t$.
The group $W$ acts on the set of $D$-tableaux in the obvious
way---if $w\in W$,
an entry $i$ is replaced by $iw$ and $tw$ denotes the tableau
resulting from the action of $w$ on the tableau $t$.
We denote by $t^{D}$ and  $t_{D}$ the two $D$-tableaux obtained by
filling the nodes of $D$ with $1,\ldots,n$ by rows and by columns,
respectively, and we write $w_{D}$ for the element of $W$ defined by
$t^{D}w_{D}=t_{D}$.

Now let $D$ be a diagram and let $t$ be a \emph{$D$-tableau}.
We say $t$ is \emph{row-standard} if it is increasing on rows.
Similarly, we say $t$ is \emph{column-standard} if it is increasing
on columns.
We say that $t$ is \emph{standard} if $(i',j')t\LEQ (i'',j'')t$
for any $(i',j'),(i'',j'')\in D$ with $i'\LEQ i''$ and $j'\LEQ j''$.
Note that a standard $D$-tableau is row-standard and column-standard,
but the converse is not true, in general.

For $1\LEQ l,m\LEQ |D|$, we write $l<_{ne}m$ (resp., $l\LEQ_{se}m$) in $t$ if $i_l>i_m$ and $j_l<j_m$ (resp., $i_l\LEQ i_m$ and $j_l\LEQ j_m$) where, for $1\LEQ r\LEQ |D|$, we set $r=(i_r,j_r)t$ with $(i_r,j_r)\in D$.
Informally, $l<_{ne}m$ means $m$ is strictly north-east of $l$ in $t$ and
$l\LEQ_{se}m$ means $m$ is weakly south-east of $l$ in $t$.

The row-form of $w\in S_n$ is obtained by writing the rows of $t^Dw$ in one row so that, for each $i$, the $(i+1)$-th row is to the right of the $i$-th row.
It follows easily that if  $k<_{ne}k+1$ in $t^Dw$, then $l(ws_k)=l(w)-1$.
Moreover, we can prove the following lemma.

\begin{lemma}\label{R31}(Compare \cite[Lemma~3.4]{MPa15}).
Let $D$ be a diagram of size $n$, let $u\in S_n$ and let $k\in\mathbb N$ with $1\LEQ k\LEQ n-1$.
Suppose that $t^Du$ is a standard $D$-tableau.
Then (i) and (ii) below hold.\\[1ex]
\begin{tabular}{lp{5.68in}}
(i) & If $l(us_k)=l(u)-1$, then $k<_{ne}k+1$ in $t^Du$, and $t^Dus_k$ is a standard $D$-tableau.
\\[1ex]
(ii) & If $k+1<_{ne}k$ in $t^Du$, then $N^-(u)\varsubsetneqq N^-(us_k)$ (so $l(us_k)=l(u)+1$) and $t^Dus_k$ is a standard $D$-tableau.
\end{tabular}
\end{lemma}
\begin{proof}
(i)
If $l(us_k)=l(u)-1$, then $k+1$ precedes $k$ in the row-form of $u$.
The assumption that $t^Du$ is standard now forces $k<_{ne}k+1$.
It also ensures that $t^Dus_k$ is standard in view of the location of $k$ and $k+1$ in $t^Du$.

\smallskip
(ii)
If $k+1<_{ne}k$ in $t^Du$, then $k$ precedes $k+1$ in the row-form of $u$ and so $ku^{-1}<(k+1)u^{-1}$.
Hence, $N^-(us_k)=N^-(u)\cup \{\epsilon_{ku^{-1}}-\epsilon_{(k+1)u^{-1}}\}$.
Again the assumption that $t^Du$ is standard, together with the location of $k$ and $k+1$ in $t^Du$, ensure that $t^Dus_k$ is standard.
\end{proof}

For a diagram $D$ of size $n$, we define the subset $\Psi_D$ of $\Phi$ by $\Psi_D=\{\epsilon_l-\epsilon_m\in\Phi\colon l\ne m$ and $l\LEQ_{se}m$ in $t^D\}$.
Clearly $\Psi_D\subseteq \Phi^+$ since $t^D$ is standard.

\begin{lemma}\label{R32}
Let $D$ be a diagram of size $n$ and let $u\in S_n$.
Suppose that $t^Du$ is a standard $D$-tableau.
Then,\\[1ex]
\begin{tabular}{lp{5.68in}}
(i) & $\Psi_D\subseteq N^+(u)$ and, moreover, $\Psi_D=N^+(w_D)$.
\\[1ex]
(ii) & If $u\ne w_D$, then $N^-(u)\varsubsetneqq N^-(us_k)\subseteq N^-(w_D)$ for some $k$ with $1\LEQ k\LEQ n-1$ and, moreover, $t^Dus_k$ is a standard $D$-tableau.
\end{tabular}
\end{lemma}
\begin{proof}
We assume the hypothesis.

(i)
Suppose that $1\LEQ l,m\LEQ n$ with $l\ne m$ and $l\LEQ_{se}m$ in $t^D$.
Since $t^D$ (resp., $t^Du$) is  standard, $l<m$ (resp., $lu<mu$).
Hence $\epsilon_l-\epsilon_m\in N^+(u)$ showing that $\Psi_D\subseteq N^+(u)$.
In particular, $\Psi_D\subseteq N^+(w_D)$ since $t_D\,(=t^Dw_D)$ is standard.

Now suppose $\alpha=\epsilon_p-\epsilon_q\in\Phi^+-\Psi_D$.
First, $p<q$ since $\alpha\in\Phi^+$.
In view of the way $t^D$ is constructed we have  $q<_{ne}p$ in $t^D$.
We also have $ qw_D<pw_D$ from the way $t_D\,(=t^Dw_D)$ is constructed.
So $\alpha w_D=\epsilon_{pw_D}-\epsilon_{qw_D}\in\Phi-\Phi^+$.
Hence, $\alpha\in N^-(w_D)$. Thus, $\Phi^+-\Psi_D\subseteq N^-(w_D)=\Phi^+-N^+(w_D)$.
Hence, $N^+(w_D)\subseteq \Psi_D$.
So $N^+(w_D)=\Psi_D$.

\smallskip
(ii)
Suppose that $u\ne w_D$.
Then $t^Du\ne t_D$, so there exists $k$ with $1\LEQ k\LEQ n-1$ such that $k+1$ appears in $t^Du$ in a column of lower index than the column $k$ appears in.
(If there is no such $k$, clearly this forces $t^Du=t_D$.)
Since $t^Du$ is standard, we must have $k+1<_{ne}k$ in $t^Du$.
The result now follows from Lemma~\ref{R31}(ii) and item~(i) of this lemma.
\end{proof}

We continue with $D$  a diagram of size $n$ and  $u$  a prefix of $w_D$.
Beginning with the standard tableau $t_D\,(=t^Dw_D)$ and applying Lemma~\ref{R31}(i) a finite number of times, we see that $t^Du$ is a standard $D$-tableau.
Conversely, if we suppose that $t^Dv$, where $v\in S_n$, is a standard $D$-tableau, a finite number of applications of Lemma~\ref{R32}(ii) shows that $v$ is a prefix of $w_D$.
(Note that by Lemma~\ref{R32}(ii) we know that $v'=w_D$ whenever $t^Dv'$ is standard and satisfies $N^-(v')=N^-(w_D)$.)
Hence we have,

\begin{result}[{\cite[Proposition 3.5]{MPa15}.
Compare \cite[Lemma 1.5]{DJa86}}]
\label{res:3a}
Let $D$ be a diagram.
Then the mapping $u\mapsto t^{D}u$ is a bijection of the set
of prefixes of $w_{D}$ to the set of standard
$D$-tableaux.
\proofblock
\end{result}

The argument presented above provides an interpretation of the proof of Result~\ref{res:3a} given in~\cite{MPa15} in terms of certain subsets of the root system $\Phi$.
Writing $\tilde N(u)=(N^-(w_D)-N^-(u))u$ for $u\in S_n$ with $t^Du$  standard,
the set $\tilde N(u)$ can naturally be identified with the set $N_u$ in the proof of~\cite[Proposition~3.5]{MPa15}.
In particular, with $D$, $u$ and $k$ as in Lemma~\ref{R32}(ii), we have $\tilde N(us_k)s_k\subseteq \tilde N(u)$ and $\tilde N(us_k)s_k=\tilde N(u)-\{\epsilon_k-\epsilon_{k+1}\}$.

\begin{remark}
(i) Lemma~\ref{R32}(ii) provides a straightforward process for completing a reduced expression of any prefix of $w_D$ to a reduced expression for $w_D$.
Compare also with~\cite[Algorithm~1]{MPa15}.
\\[1ex]
(ii)
Considering the root-subsystem of $\Phi$ corresponding to the parabolic subgroup $W_{J(\lambda_D)}$, we can easily observe that $\mathfrak{X}_{J(\lambda_D)}=\{w\in S_n\colon t^Dw$ is row-standard$\}$; see also~\cite[Lemma~1.1]{DJa86}.
In particular, $w_D$ and all its prefixes belong to $\mathfrak{X}_{J(\lambda_D)}$.
\end{remark}

In general, an element of $W$ will have an expression of the form
$w_D$ for many different diagrams $D$ of size $n$.
If $\lambda\vDash n$ and $d\in \mathfrak{X}_{J(\lambda)}$, a way to locate suitable diagrams $D\in\mathcal{D}^{(\lambda)}$  with $d=w_D$ is given in~{\cite[Proposition~3.7]{MPa15}}.
The proof involves the construction of a very particular diagram $D=D(d,\lambda)\in\mathcal D^{(\lambda)}$ with $w_D=d$.
This is formed by partitioning the row-form
of $d$ in parts of sizes corresponding to $\lambda$, placing these
parts on consecutive rows and moving the entries on the rows minimally
to make a tableau of the form $t_D$.

Moreover, in~\cite[Proposition~3.8]{MPa15} it is shown that among all diagrams $E\in\mathcal D^{(\lambda)}$ with $w_E=d$, diagram $D(d,\lambda)$ is the unique one with the minimum number of columns.
Also described in the same proposition is the way any such diagram $E$ relates to $D(d,\lambda)$.

\begin{remark}\label{rSpecial}
Let $E$ be a special diagram.
Combining Result~\ref{res:2a} with~\cite[Proposition 3.8]{MPa15}, we see that $E=D(w_E,\lambda_E)$.
\end{remark}

As in~\cite{MPa17}, for a composition $\lambda$ of $n$, we define the following subsets
of $\mathfrak{X}_{J(\lambda)}$ and $\mathcal{D}^{(\lambda)}$:
\begin{equation*}
\renewcommand{\arraystretch}{1.0}
\label{eqn:1a}
\begin{array}{rcl}
Z(\lambda) & = &
\{e\in\mathfrak{X}_{J(\lambda)}\colon w_{J(\lambda)}e\sim_Rw_{J(\lambda)}\},
\\
Z_s(\lambda) & = &
\{e\in Z(\lambda)\colon e=w_D
\mbox{ for some special diagram } D\in\mathcal{D}^{(\lambda)}\},
\\
Y(\lambda) & = &
\{x\in Z(\lambda)\colon x \mbox{ is not a prefix of any other }
y\in Z(\lambda)\},
\\
Y_s(\lambda) & = &
Y(\lambda)\cap Z_s(\lambda)=\{y\in Y(\lambda)\colon D(y,\lambda)\mbox{ is special}\},
\\
\mathcal{E}^{(\lambda)} & = &
\{D(y,\lambda)\colon y\in Y(\lambda)\}\ \mbox{ and }\
\mathcal{E}_s^{(\lambda)}=\{D\in\mathcal{E}^{(\lambda)}\colon D \mbox{ is special}\}.
\end{array}
\end{equation*}
\par
In view of Result~\ref{res:2c}, $Z(\lambda)$ is closed
under the taking of prefixes and
$w_{J(\lambda)}Z(\lambda)$ is the
right cell of $W$ containing $w_{J(\lambda)}$.
We denote this right cell by $\mathfrak{C}(\lambda)$.
A knowledge of $Y(\lambda)$ leads directly to $Z(\lambda)$
by determining all prefixes.
We call $Y(\lambda)$ the \emph{rim} of the cell
$\mathfrak{C}(\lambda)$.
The map $y\mapsto D(y,\lambda)$ from $Y(\lambda)$ to $\mathcal{E}^{(\lambda)}$ is a bijection, so $Y(\lambda)=\{w_D\colon D\in\mathcal{E}^{(\lambda)}\}$.
Hence, in order to give an explicit description of  $\mathfrak{C}(\lambda)$ it is enough to locate the diagrams in $\mathcal{E}^{(\lambda)}$.

\begin{remark}\label{rem:2.2}
In the case that $\lambda$ is a partition of $n$, it follows from
\cite[Lemma~3.3]{MPa05} that $\mathcal{E}^{(\lambda)}=\mathcal{E}_s^{(\lambda)}=\{V(\lambda)\}$.
\end{remark}

\subsection{Paths and admissible diagrams}\label{secPathsAdmDiagr}

In~\cite{MPa17} we investigated how the subsequence type of a diagram $D$, defined in Definition~\ref{def:2.1a} below, relates to the shape of
the Robinson-Schensted tableau of the element $w_{J(\lambda_D)}w_D$.
The work in~\cite{Schensted1961} and~\cite{Greene1974}, see also~\cite[Lemma~3.2]{MPa17}, motivates the following definition.
\begin{definition}[Compare with the definition before Remark~3.3
in \cite{MPa17}]
\label{def:2.1a}
Let $D$ be a diagram of size $n$.
\begin{longtable}{lp{5.68in}}
(i) &
A \emph{path} of \emph{length} $m$ in $D$ is a non-empty sequence of
nodes $((a_i,b_i))_{i=1}^m$ of $D$ such that $a_i<a_{i+1}$ and
$b_i\LEQ b_{i+1}$ for $i=1,\ldots,m-1$.
\\[1ex]
(ii) &
For $k\in\mathbb{N}$, a \emph{$k$-path} in $D$ is a sequence of $k$
mutually disjoint paths in $D$;
the paths in this sequence are the \emph{constituent paths} of the
$k$-path.
The \emph{length} of a $k$-path is the sum of the lengths of its
constituent paths; this is the total number of nodes in the $k$-path.
The \emph{type} of a $k$-path is the sequence of lengths of its paths
in non-strictly decreasing order---in particular, the type of a
$k$-path is a $k$-part partition.
The \emph{support} of a $k$-path $\Pi$, which we denote by $s(\Pi)$,
is the set of nodes occurring in its paths.
\\[1ex]
(iii) &
Let $\Pi$ be a $k$-path in  $D$ and let $k'\LEQ k$.
A \emph{$k'$-subpath} of $\Pi$ is a $k'$-path in $D$ whose constituent
paths are also constituent paths of $\Pi$.
\\[1ex]
(iv) &
A $k$-path and a $k'$-path in $D$ are said to be \emph{equivalent} to one
another if they have the same support. 
\\[1ex]
(v) &
The diagram $D$ is said to be of \emph{subsequence type}
$\nu$, where $\nu=(\nu_1,\ldots,\nu_r)\vdash n$,
if the maximum length of a $k$-path in $D$ is
$\nu_1+\ldots+\nu_k$ whenever $1\LEQ k\LEQ r$.
We call $D$ \emph{admissible} if it is of subsequence
type $\lambda_D'$.
\end{longtable}
\end{definition}
\par
See \cite[Remark~3.3]{MPa17}.
Using the notion of \emph{$k$-increasing subsequence} of the row form
of a permutation (see, for example, \cite[Definition~3.5.1]{Sagan}),
we see from \cite[Lemma~3.2 and Remark~3.3]{MPa17} that there is a bijection between the
set of $k$-paths in a diagram $D$ and the set of $k$-increasing
subsequences in $w_{J(\lambda_D)}w_D$, for any positive integer $k$.
In fact,
the increasing subsequences occurring in the row-form of
$w_{J(\lambda_D)}w_D$ are precisely the ones which have form
$((a_i,b_i)t_D)_{i=1}^m$ for some path $((a_i,b_i))_{i=1}^m$ inside $D$.
\par
As the support of a path defines a unique path, we may refer to the
path by just giving the support.
However, in general, a set of nodes may form the support of many
different $k$-paths if $k\GEQ2$.
\par
If $D,E\in\mathcal{D}^{(\lambda)}$ for some $\lambda\vDash n$, there
is a natural bijection $\theta_{E,D}\colon E\to D$ given by:
$(a,b)\theta_{E,D}$ is the $l$-th node on the $i$-th row of $D$ if
$(a,b)$ is the $l$-th node on the $i$-th row of $E$
for all nodes $(a,b)$ of $E$.
We write
$(a,b)\theta_{E,D}
=(a,(a,b){\theta_{E,D}''})$.
\begin{proposition}
\label{prop:2.3a}
Let $\lambda\vDash n$, let $D,E\in\mathcal{D}^{(\lambda)}$ and
let $\theta=\theta_{E,D}$.
Suppose that $t^Ew_D$ is a standard $E$-tableau.
If $\Pi=(\pi_1,\dots,\pi_k)$ be a $k$-path in $E$ then
$\Pi\theta=(\pi_1\theta,\dots,\pi_k\theta)$ is a $k$-path in $D$.
\end{proposition}
\begin{proof}
First, fix a particular $j$.
Let $\pi_j=((a_{i,j},b_{i,j}))_{i=1}^{m_j}$, for
$j=1,\ldots,k$.
Since $t^Ew_D$ is a standard $E$-tableau, and $\pi_j$ is a path in
$E$, $((a_{i,j},b_{i,j})t^Ew_D)_{i=1}^{m_j}$
is a strictly increasing sequence of integers.
Since $(a_{i,j},b_{i,j})t^Ew_D
=
((a_{i,j},b_{i,j})\theta)t^Dw_D
=
((a_{i,j},b_{i,j})\theta)t_D
$,
and $t_D$ is a standard $E$-tableau,
the sequence $((a_{i,j},b_{i,j})\theta)_{i=1}^{m_j}=\pi_j\theta$ of
nodes of $D$, which correspond in $t_D$ to the integers of the preceding
strictly increasing sequence, form a path in $D$.
\par
Since $\theta$ is a bijection,
$\Pi\theta=(\pi_1\theta,\dots,\pi_k\theta)$ is a $k$-path in $D$.
\end{proof}

The following observations can be made about paths and admissible diagrams.
\begin{result}[{\cite[Proposition~3.4]{MPa17}}]
\label{res:6a}
Let $D$ be a diagram and write
$\lambda_D'=(\lambda_1',\dots,\lambda_{r'}')$.
If $1\LEQ u\LEQ r'$, then a $u$-path in $D$ of length
$\sum_{1\LEQ j\LEQ u}\lambda_j'$ (if it exists) contains all $\lambda_i$ nodes on
the $i$-th row of $D$ if $\lambda_i\LEQ u$ and exactly $u$ nodes on
all remaining rows.
\par
If for each $u$, $1\LEQ u\LEQ r'$, there is a $u$-path $\Pi_u$ such
that, for all $i$, $\Pi_u$ has exactly $\min\{u,\lambda_i\}$ nodes on
the $i$-th row, then $D$ is an admissible diagram.
\proofblock
\end{result}

\begin{result}[{\cite[Propositions 3.5 and~3.6~and~Corollary~3.7]{MPa17}}]
\label{res:8a}
Let $D$ be a diagram of size $n$ and let $\nu$ be a partition of $n$.
\\[3pt]
\begin{tabular}{rp{5.5in}}
(i) & If $D$ is of subsequence type $\nu$ then
$\mu_D''\unlhd\nu\unlhd\lambda_D'$.
Moreover, $\shape (w_{J(\lambda_D)}w_D)=\nu$
if, and only if, $D$ is of subsequence type $\nu$.
\\[3pt]
(ii) & $w_{J(\lambda_D)}w_D\sim_R w_{J(\lambda_D)}$
if, and only if, $D$ is admissible. 
\par
In particular, if $D$ is a special diagram then
$D$ is admissible.
\end{tabular}
\proofblock
\end{result}

\begin{lemma}\label{BlemmaSpecial}
Let $\lambda\vDash n$ and let $y\in Y(\lambda)$.
Also let $D=D(y,\lambda)$.
Suppose that $E\in\mathcal{D}^{(\lambda)}$ is an admissible diagram with $E=D(w_E,\lambda)$ and that $t^Ew_D$ is a standard $E$-tableau.
Then $E=D$.
\end{lemma}
\begin{proof}
First observe that $w_D=y\in Y(\lambda)$.
Now $w_E\in Z(\lambda)$ since $E$ is admissible.
Since $t^Ew_D$ is a standard $E$-tableau, $w_D$ is a prefix of $w_E$
by Result~\ref{res:3a}.
Since $w_D\in Y(\lambda)$,  $w_E=w_D$.
Hence $D=D(y,\lambda)=D(w_D,\lambda)=D(w_E,\lambda)=E$.
\end{proof}

Combining Remark~\ref{rSpecial}, Result~\ref{res:8a}(ii) and Lemma~\ref{BlemmaSpecial} we get the following corollary.

\begin{corollary}\label{lemmaSpecial}
Let $\lambda\vDash n$ and let $y\in Y(\lambda)$.
Also let $D=D(y,\lambda)$.
Suppose that $E\in\mathcal{D}^{(\lambda)}$ is a special diagram and that $t^Ew_D$ is a standard $E$-tableau.
Then $E=D$.
\end{corollary}

\par
The \emph{reverse} composition $\dot\lambda$ of a composition
$\lambda=(\lambda_1$, $\dots$, $\lambda_r)$ is the composition
$(\lambda_r$, $\dots$, $\lambda_1)$ obtained by reversing the order of
the entries.
For a principal diagram $D\in\mathcal{D}^{(\lambda)}$,
the diagram $\dot D\in\mathcal{D}^{(\dot\lambda)}$ is the diagram
obtained by rotating $D$ through $180^{\circ}$.
If $D\in\mathcal{D}^{(\lambda,\mu)}$,
then $\dot D\in\mathcal{D}^{(\dot\lambda,\dot\mu)}$.
Since rotating $D$ through $180^{\circ}$ maps $k$-paths into $k$-paths,
for any $k$, diagrams $D$ and $\dot D$ have the same subsequence type.
\begin{remark}[See {\cite[Proposition~3.9]{MPa17}}]
\label{res:10a}
Let $\lambda\vDash n$ and consider the (graph) automorphism of $W$ given by $z\mapsto w_0 z w_0$ ($z\in W$) where $w_0$ is the longest element of $W$ (recall that $w_0^2$ is the identity).
Since $w_0s_iw_0=s_{n-i}$ for $1\LEQ i\LEQ n-1$, a reduced form for $w_0 z w_0$ can be obtained from a reduced form for $z$ by replacing $s_i$ with $s_{n-i}$ for $1\LEQ i\LEQ n-1$.
Moreover, if $D\in\mathcal D^{(\lambda)}$ then $t^{\dot D}$ (resp., $t_{\dot D}$) can be obtained from $t^Dw_0$ (resp., $t_Dw_0$) by rotating through $180^{\circ}$.
Since $t^Dw_0(w_0w_Dw_0)=t_Dw_0$, we get $w_{\dot D}=w_0w_Dw_0$.
We can thus easily determine the prefixes of $w_{\dot D}$ if we know the prefixes of $w_D$.
\\
A consequence of the above remarks is that the map $z\mapsto w_0zw_0$ induces an injection $Z(\lambda)\to Z(\dot\lambda)$ which restricts to injections $Y(\lambda)\to Y(\dot\lambda)$ and $Y_s(\lambda)\to Y_s(\dot\lambda)$.
Since the reverse composition of $\dot\lambda$ is the composition $\lambda$ itself, we see that the above injections are in fact bijections.
We conclude that the map $D\mapsto \dot D$ from $\mathcal D^{(\lambda,\mu)}$ to $\mathcal D^{(\dot\lambda,\dot\mu)}$ induces a bijection between the sets $\mathcal E^{(\lambda)}$ and $\mathcal E^{(\dot\lambda)}$.
\end{remark}
\par
Note that if $\rho$ is the representation of $S_n$ corresponding to
the cell $\mathfrak{C}(\lambda)$, then the representation corresponding
to $\mathfrak{C}(\dot\lambda)$ is given by $s_i\mapsto s_{n-i}\rho$
for all $i$.
\par
We now see that partitions and their reverses are the only
compositions for which the corresponding cells consist of prefixes
of a single element.
%
%
\begin{theorem}\label{thm:2.16a}
Let $\lambda\vDash n$.
Then $|Y(\lambda)|=1$ if, and only if, either $\lambda$ or
$\dot\lambda$ is a partition.
\end{theorem}
\begin{proof}
The `if' part follows from Remarks~\ref{rem:2.2} and~\ref{res:10a}. 
For the `only if' part, let $\lambda=(\lambda_1,\dots,\lambda_r)$
and suppose that neither $\lambda$ nor $\dot\lambda$ is a partition.
In view of Remark~\ref{res:10a}, we may replace $\lambda$ by
$\dot\lambda$ if necessary and assume that either
(i)~$\lambda_b<\lambda_a\LEQ\lambda_c$ or
(ii)~$\lambda_c\LEQ\lambda_a<\lambda_b$
for some $a$, $b$ and $c$ with
$1\LEQ a<b<c\LEQ r$.
Let $l$ be the number of parts of $\lambda'$ and, for convenience,
let $\mu=\lambda'$, $d=\lambda_a$, $e=\lambda_b$, and
$f=\lambda_c$.
Let $F_1$ be the unique diagram in $\mathcal{D}^{(\lambda,\mu)}$, and
let $F_2$ be the unique diagram in $\mathcal{D}^{(\lambda,\dot\mu)}$.
By Result~\ref{res:8a}(ii), $w_{F_1},w_{F_2}\in Z(\lambda)$ since $F_1$, $F_2$ (which belong to $\mathcal D^{(\lambda)}$) are special diagrams.
\par
Suppose now that $w_{F_1}$ and $w_{F_2}$ are prefixes of a single
$y\in Y(\lambda)$.
Then $y=w_G$, where $G\in\mathcal{D}^{(\lambda)}$ is admissible (see Result~\ref{res:8a}(ii)).
By Result~\ref{res:3a}, $t^Gw_{F_1}$ and $t^Gw_{F_2}$ are standard $G$-tableaux.
\par
Let
$\{(a,p_1),\dots,(a,p_d)\}$, $\{(b,q_1),\dots,(b,q_e)\}$,
$\{(c,r_1),\dots,(c,r_f)\}$ be the sets of nodes of $G$ on its
$a$-th row, $b$-th row, and $c$-th row, respectively,
where $p_1<\dots<p_d$, $q_1<\dots<q_e$, and $r_1<\dots<r_f$.
%
\par
\emph{Case (i):} $e<d\LEQ f$.
\par
Since $(b,e)t_{F_1}<(a,e+1)t_{F_1}$, we get
$((b,e)\theta_{F_1,G})t^Gw_{F_1}<((a,e+1)\theta_{F_1,G})t^Gw_{F_1}$.
Hence $(b,q_e)t^Gw_{F_1}<(a,p_{e+1})t^Gw_{F_1}$.
As $t^Gw_{F_1}$ is standard, $q_e<p_{e+1}$.
\par
Similarly,
$(c,l-e)t_{F_2}<(b,l-e+1)t_{F_2}$.
So,
$((c,l-e)\theta_{F_2,G})t^Gw_{F_2}<((b,l-e+1)\theta_{F_2,G})t^Gw_{F_2}$.
Hence
$(c,r_{f-e})t^Gw_{F_2}<(b,q_1)t^Gw_{F_2}$.
As $t^Gw_{F_2}$ is standard, $r_{f-e}<q_1$.
\par
Since $G$ has subsequence type $\lambda'$ and $e+1\LEQ l$,
there exists an $(e+1)$-path $\Pi$ in $G$ of length $\lambda_1'+\ldots+\lambda_{e+1}'$.
By Result~\ref{res:6a}, $\Pi$ contains exactly $e+1$ nodes on each
of the $a$-th and $c$-th rows of $G$ and contains all $e$ nodes on
the $b$-th row of $G$.
\par
Let $\Pi'$ be the $e$-subpath of $\Pi$ containing all the nodes on the
$b$-th row of $G$.
The remaining path $\pi$ of $\Pi$ necessarily contains one node on
each of the $a$-th and $c$-th rows.
Since $q_e<p_{e+1}$, the nodes of $\Pi'$ on the $a$-th row are in
columns $\LEQ q_e$.
Hence they are the nodes $(a,p_i)$ with $1\LEQ i\LEQ e$.
So, the node of $\pi$ on the $a$-th row is $(a,p_{i'})$ for some
$i'\GEQ e+1$.
Since each one of the constituent paths of $\Pi'$ contains precisely one node on each
of rows $b$ and $c$ of $G$ and $\Pi'$ contains all nodes on row $b$ of $G$, the nodes of $\Pi'$ on the $c$-th row of $G$ are in
columns $j\GEQ q_1$.
Hence they are the nodes $(c,r_j)$ with $f-e+1\LEQ j\LEQ f$ in view of the
fact that $r_{f-e}<q_{1}$.
So, the node of $\pi$ on the $c$-th row is $(c,r_{j'})$ for some
$j'\LEQ f-e$.
Since $r_{j'}\LEQ r_{f-e}<q_{1}\LEQ q_e<p_{e+1}\LEQ p_{i'}$,
this is impossible.
\par\vspace{1ex}
\emph{Case (ii):}
$f\LEQ d<e$. 
\par
Since $(c,f)t_{F_1}<(b,f+1)t_{F_1}$, we get
$((c,f)\theta_{F_1,G})t^Gw_{F_1}<((b,f+1)\theta_{F_1,G})t^Gw_{F_1}$.
Hence $(c,r_f)t^Gw_{F_1}<(b,q_{f+1})t^Gw_{F_1}$.
Since $t^Gw_{F_1}$ is standard, $r_f< q_{f+1}$.
\par
Similarly,
$(b,l\!-\!d)t_{F_2}<(a,l\!-\!d\!+\!1)t_{F_2}$.
So,
$((b,l\!-\!d)\theta_{F_2,G})t^Gw_{F_2}<((a,l\!-\!d\!+\!1)\theta_{F_2,G})t^Gw_{F_2}$.
Hence
$(b,q_{e-d})t^Gw_{F_2}<(a,p_1)t^Gw_{F_2}$.
As $t^Gw_{F_2}$ is standard, $q_{e-d}<p_1$.
\par
Next we let $\tilde\Pi$ be an $f$-path in $G$ of length
$\lambda_1'+\ldots+\lambda_f'$.
By Result~\ref{res:6a}, $\tilde\Pi$ must contain precisely $f$ nodes
on each one of rows $a$, $b$, $c$ of $G$; in particular it
contains all nodes on row $c$ of $G$.
It follows that each of the constituent paths of $\tilde\Pi$ contains
precisely one node on each one of rows $a$, $b$ and $c$ of $G$.
\par
Since $r_f<q_{f+1}$, the nodes $(b,q_j)$ for $j\GEQ f+1$ cannot belong
to $\tilde\Pi$.
So the nodes in row $b$ of $G$ which belong to $\tilde\Pi$ are
precisely the nodes $(b,q_j)$, $1\LEQ j\LEQ f$.
Now let $p_{j_0}$ be the minimum of the column indices of the nodes
in $\tilde\Pi$ that belong to row $a$ of $G$.
We then have, $q_1\LEQ q_{e-d}<p_1\LEQ p_{j_0}$.
It follows that node $(b,q_1)$ does not belong to $\tilde\Pi$,
giving the desired contradiction.
\par
As neither of the cases is possible, the assumption that $w_{F_1}$
and $w_{F_2}$ are prefixes of a single element of $Y(\lambda)$ is
false.
So $|Y(\lambda)|\GEQ 2$.
\end{proof}

\end{section}

\begin{section}{Ordered $k$-paths}\label{sec:OrdPaths}

In this section we introduce the notion of an ordered $k$-path.
This will play a key role in determining the rim for certain families of Kazhdan-Lusztig cells (see Section~\ref{sec:3a}).

We first consider a relation between non-empty subsets of a diagram $D$.

\begin{definition}
\label{def:2.5a}
Let $D$ be a diagram and let $D_1$, $D_2$ be non-empty
subsets of $D$.
We write ``$D_1\prec D_2$'' if whenever $(a_1,b_1)$ and
$(a_2,b_2)$ are nodes of $D_1$ and $D_2$ respectively with
$a_1\le a_2$, then $b_1<b_2$; otherwise we write ``$D_1\not\prec D_2$''.
Note that $D_1\not\prec D_2$ if $D_1\cap D_2\ne\varnothing$.
In particular, $E\not\prec E$ for any non-empty subset $E$ of $D$.
\end{definition}
\begin{example}
\label{ex:2.7a}
In this picture of a non-empty subset $E$ of a diagram $D$, only the nodes of $D$ which belong to $E$ are included and they are represented by
small black discs.
The  distance between the two nodes of lower column index in the second row of $E$  is one unit.
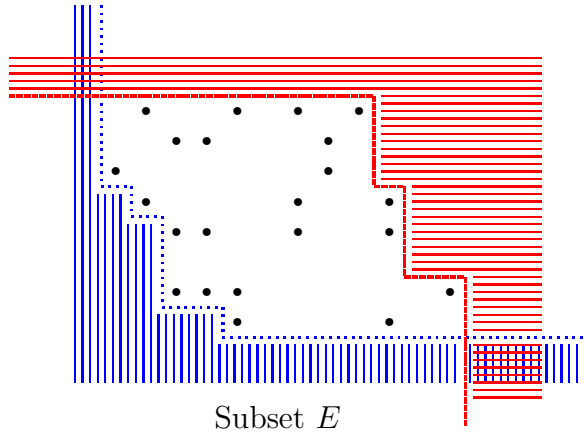
\begin{center}
{\setlength{\unitlength}{1mm}
\begin{picture}(57,60)(0,-50)
%
%
{\color{blue}
\put(3,12){\line(0,-1){50}}
\put(4,12){\line(0,-1){50}}
\put(5,12){\line(0,-1){50}}
\put(6,-13){\line(0,-1){25}}
\put(7,-13){\line(0,-1){25}}
\put(8,-13){\line(0,-1){25}}
\put(9,-13){\line(0,-1){25}}
\put(10,-17){\line(0,-1){21}}
\put(11,-17){\line(0,-1){21}}
\put(12,-17){\line(0,-1){21}}
\put(13,-17){\line(0,-1){21}}
\put(14,-29){\line(0,-1){9}}
\put(15,-29){\line(0,-1){9}}
\put(16,-29){\line(0,-1){9}}
\put(17,-29){\line(0,-1){9}}
\put(18,-29){\line(0,-1){9}}
\put(19,-29){\line(0,-1){9}}
\put(20,-29){\line(0,-1){9}}
\put(21,-29){\line(0,-1){9}}
\put(22,-33){\line(0,-1){5}}
\put(23,-33){\line(0,-1){5}}
\put(24,-33){\line(0,-1){5}}
\put(25,-33){\line(0,-1){5}}
\put(26,-33){\line(0,-1){5}}
\put(27,-33){\line(0,-1){5}}
\put(28,-33){\line(0,-1){5}}
\put(29,-33){\line(0,-1){5}}
\put(30,-33){\line(0,-1){5}}
\put(31,-33){\line(0,-1){5}}
\put(32,-33){\line(0,-1){5}}
\put(33,-33){\line(0,-1){5}}
\put(34,-33){\line(0,-1){5}}
\put(35,-33){\line(0,-1){5}}
\put(36,-33){\line(0,-1){5}}
\put(37,-33){\line(0,-1){5}}
\put(38,-33){\line(0,-1){5}}
\put(39,-33){\line(0,-1){5}}
\put(40,-33){\line(0,-1){5}}
\put(41,-33){\line(0,-1){5}}
\put(42,-33){\line(0,-1){5}}
\put(43,-33){\line(0,-1){5}}
\put(44,-33){\line(0,-1){5}}
\put(45,-33){\line(0,-1){5}}
\put(46,-33){\line(0,-1){5}}
\put(47,-33){\line(0,-1){5}}
\put(48,-33){\line(0,-1){5}}
\put(49,-33){\line(0,-1){5}}
\put(50,-33){\line(0,-1){5}}
\put(51,-33){\line(0,-1){5}}
\put(52,-33){\line(0,-1){5}}
\put(53,-33){\line(0,-1){5}}
\put(55,-33){\line(0,-1){5}}
\put(56,-33){\line(0,-1){5}}
\put(57,-33){\line(0,-1){5}}
\put(58,-33){\line(0,-1){5}}
\put(59,-33){\line(0,-1){5}}
\put(60,-33){\line(0,-1){5}}
\put(61,-33){\line(0,-1){5}}
\put(62,-33){\line(0,-1){5}}
\put(63,-33){\line(0,-1){5}}
\put(64,-33){\line(0,-1){5}}
\put(65,-33){\line(0,-1){5}}
\put(66,-33){\line(0,-1){5}}
\put(67,-33){\line(0,-1){5}}
\put(68,-33){\line(0,-1){5}}
\put(69,-33){\line(0,-1){5}}
 \multiput(6.5,12)(0,-1){12}{\thicklines\line(0,-1){0.25}}
 \multiput(6.5,0)(0,-1){12}{\thicklines\line(0,-1){0.25}}
 \multiput(6.5,-12)(1,0){4}{\thicklines\line(1,0){0.25}}
 \multiput(10.5,-12)(0,-1){4}{\thicklines\line(0,-1){0.25}}
 \multiput(10.5,-16)(1,0){4}{\thicklines\line(1,0){0.25}}
 \multiput(14.5,-16)(0,-1){12}{\thicklines\line(0,-1){0.25}}
 \multiput(14.5,-28)(1,0){8}{\thicklines\line(1,0){0.25}}
 \multiput(22.5,-28)(0,-1){4}{\thicklines\line(0,-1){0.25}}
 \multiput(22.5,-32)(1,0){48}{\thicklines\line(1,0){0.25}}
}
%
%
{\color{red}
\put(-7,5){\line(1,0){70}}
\put(-7,4){\line(1,0){70}}
\put(-7,3){\line(1,0){70}}
\put(-7,2){\line(1,0){70}}
\put(-7,1){\line(1,0){70}}
\put(42,0){\line(1,0){21}}
\put(42,-1){\line(1,0){21}}
\put(42,-2){\line(1,0){21}}
\put(42,-3){\line(1,0){21}}
\put(42,-4){\line(1,0){21}}
\put(42,-5){\line(1,0){21}}
\put(42,-6){\line(1,0){21}}
\put(42,-7){\line(1,0){21}}
\put(42,-8){\line(1,0){21}}
\put(42,-9){\line(1,0){21}}
\put(42,-10){\line(1,0){21}}
\put(42,-11){\line(1,0){21}}
\put(46,-12){\line(1,0){17}}
\put(46,-13){\line(1,0){17}}
\put(46,-14){\line(1,0){17}}
\put(46,-15){\line(1,0){17}}
\put(46,-16){\line(1,0){17}}
\put(46,-17){\line(1,0){17}}
\put(46,-18){\line(1,0){17}}
\put(46,-19){\line(1,0){17}}
\put(46,-20){\line(1,0){17}}
\put(46,-21){\line(1,0){17}}
\put(46,-22){\line(1,0){17}}
\put(46,-23){\line(1,0){17}}
\put(54,-24){\line(1,0){9}}
\put(54,-25){\line(1,0){9}}
\put(54,-26){\line(1,0){9}}
\put(54,-27){\line(1,0){9}}
\put(54,-28){\line(1,0){9}}
\put(54,-29){\line(1,0){9}}
\put(54,-30){\line(1,0){9}}
\put(54,-31){\line(1,0){9}}
\put(54,-33){\line(1,0){9}}
\put(54,-34){\line(1,0){9}}
\put(54,-35){\line(1,0){9}}
\put(54,-36){\line(1,0){9}}
\put(54,-37){\line(1,0){9}}
\put(54,-38){\line(1,0){9}}
\put(54,-39){\line(1,0){9}}
\put(54,-40){\line(1,0){9}}
 \multiput(-7,0)(1,0){48}{\thicklines\line(1,0){0.75}}
 \multiput(41,0)(0,-1){12}{\thicklines\line(0,-1){0.75}}
 \multiput(41,-12)(1,0){4}{\thicklines\line(1,0){0.75}}
 \multiput(45,-12)(0,-1){12}{\thicklines\line(0,-1){0.75}}
 \multiput(45,-24)(1,0){8}{\thicklines\line(1,0){0.75}}
 \multiput(53,-24)(0,-1){8}{\thicklines\line(0,-1){0.75}}
 \multiput(53,-32)(0,-1){12}{\thicklines\line(0,-1){0.75}}
}
%
%
\nodeB{9}{0}{4}
\nodeB{21}{0}{4}
\nodeB{29}{0}{4}
\nodeB{37}{0}{4}

\nodeB{13}{-4}{4}
\nodeB{17}{-4}{4}
\nodeB{33}{-4}{4}

\nodeB{5}{-8}{4}
\nodeB{33}{-8}{4}

\nodeB{9}{-12}{4}
\nodeB{29}{-12}{4}
\nodeB{41}{-12}{4}

\nodeB{13}{-16}{4}
\nodeB{17}{-16}{4}
\nodeB{29}{-16}{4}
\nodeB{41}{-16}{4}

\nodeB{13}{-24}{4}
\nodeB{17}{-24}{4}
\nodeB{21}{-24}{4}
\nodeB{49}{-24}{4}

\nodeB{21}{-28}{4}
\nodeB{41}{-28}{4}
\put(20,-44){\mbox{Subset $E$}}
\end{picture}
}
\end{center}
An arbitrary node $(a,b)$ of $D$ satisfies $\{(a,b)\}\prec E$ if, and only
if, $(a,b)$ is to the left and below the dotted line
\mbox{\setlength{\unitlength}{1mm}
(\begin{picture}(11,0)(0,-1)
\color{blue}
\multiput(0,0)(1,0){12}{\thicklines\line(1,0){0.25}}
\end{picture}),}
which partitions $\mathbb{Z}^2$ into two infinite regions.
Similarly, an arbitrary node $(a,b)$ of $D$ satisfies $E\prec \{(a,b)\}$
if, and only if, $(a,b)$ is to the right and above the dashed line
\mbox{\setlength{\unitlength}{1mm}
(\begin{picture}(11,0)(0,-1)
\color{red}
\multiput(0,0)(1,0){12}{\thicklines\line(1,0){0.75}}
\end{picture}),}
which partitions $\mathbb{Z}^2$ into two infinite regions.
\par
We make this more precise in the following definitions and lemma.
\end{example}

\begin{definition}
\label{def:2.8a}
Let $E$ be a non-empty subset of a diagram $D$.
Let $c_D=\sup\{b\in\mathbb{N}\colon (a,b)\in D\}$ and, for each $n\in\mathbb N$, let $E(\LEQ n)=\{b\in\mathbb{N}\colon (a,b)\in E\mbox{ for some }a\LEQ n\}$ and $E(\GEQ n)=\{b\in\mathbb{N}\colon (a,b)\in E\mbox{ for some }a\GEQ n\}$.
Define $\gamma_E(n)=\sup E(\LEQ n)$ if $E(\LEQ n)\ne\varnothing$ (resp., $\gamma_E(n)=0$ if $E(\LEQ n)=\varnothing$) and $\delta_E(n)=\inf E(\GEQ n)$ if $E(\GEQ n)\ne\varnothing$ (resp., $\delta_E(n)=1+c_D$ if $E(\GEQ n)=\varnothing$).
We also define the \emph{right side} $\mathcal{R}(E)$ and the \emph{left side}
$\mathcal{L}(E)$ of $E$ to be the sets
$\mathcal{R}(E)=\{(n,m)\in\mathbb{N}^2\colon m>\gamma_E(n)\}$
and
$\mathcal{L}(E)=\{(n,m)\in\mathbb{N}^2\colon m<\delta_E(n)\}$.
\end{definition}
\begin{lemma}
\label{lem:2.9a}
If $D_1$ and $D_2$ are non-empty subsets of a diagram $D$,
then the following statements are equivalent:
\[
\begin{tabular}{cccccccccc}
(i) \ $D_1\subseteq\mathcal{L}(D_2)$, &
(ii) \ $D_2\subseteq\mathcal{R}(D_1)$, &
(iii) \ $D_1\prec D_2$.
\end{tabular}
\]
\end{lemma}
\begin{proof}
(i)$\Rightarrow$(ii):
From (i), if $(a,b)\in D_1$ then $(a,b)\in\mathcal{L}(D_2)$.
So $b<\delta_{D_2}(a)$.
That is, $b<d$ if $(c,d)\in D_2$ and $a\leq c$.
\par
Hence, given $(c,d)\in D_2$, $d>b$ for all nodes $(a,b)\in D_1$ with
$a\leq c$.
That is, $d>\gamma_{D_1}(c)$.
So, $(c,d)\in\mathcal{R}(D_1)$.
Thus, $D_2\subseteq\mathcal{R}(D_1)$.
\par
(ii)$\Rightarrow$(iii):
From (ii), if $(c,d)\in D_2$ then $(c,d)\in\mathcal{R}(D_1)$.
So $d>\gamma_{D_1}(c)$.
That is, $b<d$ if $(a,b)\in D_1$ and $a\leq c$.
Hence, if $(a,b)\in D_1$ and $(c,d)\in D_2$ and $a\leq c$ then $b<d$.
That is, $D_1\prec D_2$.
\par
(iii)$\Rightarrow$(i):
Let $(a,b)\in D_1$.
If $(c,d)\in D_2$ with $a\leq c$, then by (iii), $b<d$.
Hence, $b<\delta_{D_2}(a)$.
So, $(a,b)\in \mathcal{L}(D_2)$.
Hence, $D_1\subseteq\mathcal{L}(D_2)$.
\end{proof}
\begin{remark}
\label{rem:2.10a}
We collect some immediate observations  regarding the relation $\prec$.
Let $D$ be a diagram and let $D_1$, $D_2$, $D_3$ be
non-empty subsets of $D$.
\par
\begin{tabular}{rp{5.5in}}
(i) & If $D_1'$ and $D_2'$ are non-empty subsets of $D_1$ and $D_2$, respectively, and $D_1\prec D_2$
 then $D_1'\prec D_2'$.
\\[1ex]
(ii) &
If $D_1\prec D_3$ and $D_2\prec D_3$, then $(D_1\cup D_2)\prec D_3$.
\\[1ex]
(iii) & If $D_1\prec D_2$ and $D_2\prec D_3$, it is not true in
general that $D_1\prec D_3$.
For example, take $D_1=\{(2,3)\}$, $D_2=\{(1,1)\}$ and $D_3=\{(3,2)\}$.
\\[1ex]
(iv) &
It is possible for both $D_1\prec D_2$ and $D_2\prec D_1$ to be true.
For example, take $D_1=\{(1,1)\}$ and $D_2=\{(2,2)\}$.
\end{tabular}
\end{remark}

We now give the definition of an ordered $k$-path in a diagram $D$.

\begin{definition}
\label{rem:2.6a}
Let $\Pi=(\pi_1,\ldots,\pi_k)$ be a $k$-path in a diagram $D$.
We say that $\Pi$ is \emph{ordered} if $s(\pi_i)\prec s(\pi_j)$ whenever $i<j$ ($1\LEQ i,j\LEQ k$).
It is then immediate that $\Pi$ is ordered if, and only if, $\bigcup_{i=1}^{j-1}s(\pi_i)\prec s(\pi_j)$ for
$2\le j\le k$.
\end{definition}

\begin{remark}
\label{rem:2.4a}
Keeping the setup and notation  of Proposition~\ref{prop:2.3a},
it does not follow, in general, that if $\Pi$ is an ordered $k$-path
then $\Pi\theta$ is an ordered $k$-path.
Consider for example
\[
\newcommand{\clrr}[1]{#1}
\newcommand{\clrg}[1]{#1}
\newcommand{\clrm}[1]{#1}
\newcommand{\clrb}[1]{#1}
\begin{array}{ccc}
t^Ew_D =
\begin{array}{rrrrrr}
\clrb{1} & \clrr{4} \\
\clrb{2} &   &   & \clrg{5} & \clrm{8} \\
\clrb{3} &   & \clrr{7} &   & \clrg{9} & \clrm{10} \\
  & \clrb{6} \\
\end{array}
& &
t_D =
\begin{array}{rrrrrr}
\clrb{1} & \clrr{4} \\
\clrb{2} & \clrg{5} &   & \clrm{8} \\
\clrb{3} &   & \clrr{7} & \clrg{9} & \clrm{10} \\
  & \clrb{6} \\
\end{array}
\end{array}
\]
The partition $\{\{1,2,3,6\},\{4,7\},\{5,9\},\{8,10\}\}$ of the
entries in $t^Ew_D$ gives an ordered $4$-path in $E$,
while the corresponding $4$-path in $D$ is not ordered.
\end{remark}

However, we can prove the following lemma.
We will need a definition first.

\begin{definition}\label{dBasic}
Let $\lambda\vDash n$ and let $D\in \mathcal D^{(\lambda)}$.
Suppose that $D=s(\Pi)$ for some ordered $k$-path $\Pi$ with $\Pi=(\pi_1,\ldots,\pi_k)$.
We denote by $D(\Pi)$ the diagram in $\mathcal D^{(\lambda)}$  constructed from $D$ by replacing each node of $\pi_j$ by a node on the same row but in the column $j$, for $j=1,\ldots, k$.
Observe that carrying out the same operation on the tableau $t_D$ results in the $D(\Pi)$-tableau $t^{D(\Pi)}w_D$.
\end{definition}

\begin{lemma}\label{lBasic}
Let $\lambda\vDash n$ and let $D\in\mathcal D^{(\lambda)}$.
Suppose that $D=s(\Pi)$ for some ordered $k$-path $\Pi$.
Then $t^{D(\Pi)}w_D$ is a standard $D(\Pi)$-tableau.
\end{lemma}

\begin{proof}
Let $\Pi=(\pi_1,\ldots,\pi_k)$ and set $E=D(\Pi)$.
Let $(a,j)$ and $(g,j')$ be nodes of $E$ with $a\LEQ g$ and $j\LEQ j'$,
and let $(a,b)$ and $(g,h)$ be the corresponding nodes of $D$.
(Note that $(a,b)=(a,j)\theta^{-1}_{D,E}$ and $(b,h)=(g,j')\theta^{-1}_{D,E}$,  with the bijection $\theta_{D,E}:D\to E$ as defined in Section~\ref{secPathsAdmDiagr}.)
Then $(a,b)$ is on $\pi_j$ and $(g,h)$ is on $\pi_{j'}$.
We claim that $b\LEQ h$.
To see this, we can use the fact that $\pi_j$ is a path for the case $j=j'$, and the fact that the $k$-path $\Pi$ is ordered, and so $s(\pi_j)\prec s(\pi_{j'})$, for the case $j<j'$.
Hence, $(a,j)t^Ew_D=(a,b)t_D\LEQ (g,h)t_D=(g,j')t^Ew_D$.
This shows that  $t^Ew_D$ is a standard $E$-tableau.
\end{proof}

It will be convenient to prove some further elementary results
concerning the relation $\prec$.
\begin{lemma}
\label{lem:2.11a}
Let $\pi$ be a path in a diagram $D$ and let $(a',b')$ be a node of $D$
which is not in $s(\pi)$ and such that $s(\pi)$ does not have two
nodes $(a_1,b')$ and $(a_2,b')$ with $a_1<a'<a_2$.
Let $\pi'$ be the path of length one with $s(\pi')=\{(a',b')\}$.
Then either $s(\pi')\prec s(\pi)$ or $s(\pi)\prec s(\pi')$.
\end{lemma}
\begin{proof}
Suppose that $s(\pi')\not\prec s(\pi)$.
Then there is a node $(c,d)\in s(\pi)$ such that $a'\LEQ c$ and
$b'\GEQ d$.
Let $(a,b)\in s(\pi)$ with $a\LEQ a'$.
Since $a\LEQ c$, $b\LEQ d\LEQ b'$.
If $b=b'$ then $b'=d$.
So $a\neq a'$ and $a'\neq c$.
Since $(a',b')$ is between the nodes $(a,b')$ and $(c,b')$ of
$\pi$, this contradicts the hypothesis.
Hence $b<b'$ and we have shown that $s(\pi)\prec s(\pi')$.
\end{proof}
\begin{lemma}
\label{lem:2.12a}
Let $(\pi_1,\pi_2)$ be an ordered 2-path in a diagram $D$ and let
$\pi'$ be a path of length one in $D$ such that
$s(\pi')\cap(s(\pi_1)\cup s(\pi_2))=\varnothing$ and
$s(\pi_1)\not\prec s(\pi')$.
Then $s(\pi')\prec s(\pi_2)$.
\end{lemma}
\begin{proof}
Suppose that $s(\pi')\not\prec s(\pi_2)$.
Let $s(\pi')=\{(a',b')\}$.
There is a node $(c,d)\in s(\pi_2)$ such that $a'\LEQ c$ and
$b'\GEQ d$.
Also, by hypothesis, there is a node $(a,b)\in s(\pi_1)$ such that
$a\LEQ a'$ and $b\GEQ b'$.
Hence, $a\LEQ c$ and $b\GEQ d$.
However, since $(\pi_1,\pi_2)$ is an ordered 2-path and $a\LEQ c$,
we get $b<d$.
This contradiction establishes the result.
\end{proof}
\par
\begin{remark}
\label{rem:2.13a}
Let $\Pi$ be a path with at least $k$ nodes in a diagram $D$.
Then $\Pi$ is equivalent to an ordered $k$-path $\Pi'$ in $D$.
For example, if $s(\Pi)=\{(a_1,b_1),\ldots,(a_l,b_l)\}$ where
$a_1<\cdots<a_l$ and $k\LEQ l$, let $\Pi'=(\pi_1,\ldots,\pi_k)$ with
$s(\pi_k)=\{(a_1,b_1),\ldots,(a_{l-k+1},b_{l-k+1})\}$ and
$s(\pi_i)=\{(a_{l-i+1},b_{l-i+1})\}$ for $1\LEQ i<k$.
\end{remark}
\begin{theorem}
\label{thm:2.14a}
Let $k\ge1$ and suppose $\Pi$ is a $k$-path in a diagram $D$.
Then $\Pi$ is equivalent to an ordered $k$-path in $D$.
\end{theorem}
\begin{proof}
In view of Remark~\ref{rem:2.13a}, it is enough to construct an ordered
$k'$-path $\Pi'=(\rho_{k'},\ldots,\rho_1)$ in $D$ where $1\le k'\le k$,
with $s(\Pi)=s(\Pi')$.
\par
Our first task is to construct a path $\rho_1$ and we do this in a
sequence of steps.
\par
%
%
\emph{First construction:}
\par
\hspace{1ex}
\parbox{6in}{%
Set $\Pi_0=s(\Pi)$, $\pi_0=\varnothing$;
also $\Pi_0\neq\varnothing$ by the definition of a path.
We construct two sequences $\{\Pi_r\}_{r\geq0}$ and
$\{\pi_r\}_{r\geq0}$ of sets of nodes of $D$.
\par
Let $r\ge1$ and assume that $\Pi_{r-1}$ and $\pi_{r-1}$ have been
constructed.
\par
\begin{tabular}{rp{5.5in}}
(A0) &
If $\Pi_{r-1}=\varnothing$, we terminate both sequences.
\\[1ex]
(A1) &
If $\Pi_{r-1}\ne\varnothing$, let $i_r$ be the least index of a row
in $D$ with a node in $\Pi_{r-1}$.
\\[1ex]
(A2) &
Let $x_r=(i_r,j_r)$ be the node in $\Pi_{r-1}$ with greatest column
index $j_r$.
\\[1ex]
(A3) &
If $\pi_{r-1}$ has no node  with column index greater than $j_r$,
let $\pi_r=\pi_{r-1}\cup\{x_r\}$; otherwise, let $\pi_r=\pi_{r-1}$.
\\[1ex]
(A4) &
Obtain $\Pi_r$ from $\Pi_{r-1}$ by removing all nodes in it with
row-index $i_r$.
\\[1ex]
(A5) &
Replace $r$ be $r+1$ and repeat (A0)-(A5).
\end{tabular}
\par\vspace{1ex}
In this process, $|\Pi_{r}|<|\Pi_{r-1}|$ if $|\Pi_{r-1}|>0$.
Hence, the process must terminate.
Also, since $\Pi_0\neq\varnothing$, the node $x_1=(i_1,j_1)$ exists
and $\pi_1=\{x_1\}\neq\varnothing$.
Let $\rho_1=\bigcup_{r\ge1}\pi_r$.
}
\par
By the first construction, $\rho_1$ is a path in $\Pi_0$.
If $\Pi_0=s(\rho_1)$, we can set $\Pi'=(\rho_1)$ and obtain the
result.
So we may assume $\Pi_0-s(\rho_1)\ne\varnothing$.
We show that $\Pi_0-s(\rho_1)\prec s(\rho_1)$.
\par
Let $(c_1,d_1)\in\Pi_0-s(\rho_1)$ and $(c_2,d_2)\in s(\rho_1)$ with
$c_1\le c_2$.
If $c_1=c_2$, then $d_1<d_2$ from the construction.
If $c_1<c_2$, there exists a node $(c_0,d_0)$ in $\rho_1$ with
$c_0\le c_1(<c_2)$ and $d_0>d_1$ from the construction.
Now $(c_0,d_0)$, $(c_2,d_2)$ are nodes of $\rho_1$ with $c_0<c_2$.
Since $\rho_1$ is a path, $d_0\le d_2$.
But $d_1<d_0$, so $d_1<d_2$ in this case also.
We conclude that $\Pi_0-s(\rho_1)\prec s(\rho_1)$.
\par
%
%
\emph{Second construction:}
\par
\hspace{1ex}
\parbox{6in}{%
Now let $P_0=P_1=\Pi_0$ and recall that $s(\rho_1)\subseteq P_1$.
We construct two sequences $\{P_q\}_{q\geq1}$ and $\{\rho_q\}_{q\geq1}$
where $P_q$ is a set of nodes of $D$ and $\rho_q$ is a path in $D$ with $s(\rho_q)\subseteq P_q$,
for $q\geq1$.

Let $r>1$ and assume that $P_{r-1}$ and $\rho_{r-1}$ have been
constructed so that $P_{r-1}\subseteq P_{r-2}$ and
$\rho_{r-1}$ is a path with $s(\rho_{r-1})\subseteq P_{r-1}$.
}
\par
\hspace{1ex}
\parbox{6in}{%
\begin{tabular}{rp{5.5in}}
(B0) &
Let $P_{r}=P_{r-1}-s(\rho_{r-1})$.
\\[1ex]
(B1) &
If $P_{r}=\varnothing$ then let $p=r-1$ and let $\Pi''$ be the
$p$-path $(\rho_{p},\rho_{p-1},\ldots,\rho_{1})$.
\\[1ex]
(B2) &
If $P_{r}\neq\varnothing$ then $P_{r}\prec s(\rho_{r-1})$ and
we construct the path $\rho_{r}$ with $s(\rho_r)\subseteq P_{r}$ using the first
construction.
\\[1ex]
(B3) &
Replace $r$ by $r+1$ and repeat steps (B0)-(B3).
\\[1ex]
\end{tabular}
}
\par\vspace{1ex}
In this process, if $|P_{r-1}|>0$ then $|P_{r}|<|P_{r-1}|$
and, since $s(\rho_{q-1})=P_{q-1}-P_{q}$ for $1<q\LEQ r$, the paths
$s(\rho_1),\ldots,s(\rho_{r-1})$ are mutually disjoint.
Also, for $1<q<r$, $s(\rho_{r-1})\subseteq P_{r-1}\subseteq P_{q}$ and,
since $P_{q}\prec s(\rho_{q-1})$, we get
 $s(\rho_{r-1})\prec s(\rho_{q-1})$.
Since the sizes of the sets in the sequence $\{P_r\}$ are strictly
decreasing, the process must terminate and the $p$-path $\Pi''$ is
an ordered $p$-path which is equivalent to $\Pi$.
\par
If $p\LEQ k$, the remarks at the start of the proof complete the
proof.
So we can assume that $p>k$.
We can easily deduce from this assumption that $P_{k+1}\ne\varnothing$ or, equivalently, that $P_k$ is not the support of a path.
It follows that there are nodes $(a_k,b_k)$ and $(a_{k+1},b_{k+1})$ in
$P_k$ such that $a_k\LEQ a_{k+1}$ and $b_k>b_{k+1}$.
Now node $(a_k,b_k)$ is not a node of $\rho_{k-1}$.
Hence, by the first construction, there is a node $(a_{k-1},b_{k-1})$
of path $\rho_{k-1}$ such that $a_{k-1}\LEQ a_k$ and $b_{k-1}>b_k$,
since $(a_k,b_k)$ was not picked in forming $\rho_{k-1}$.
Be repeating this argument we can find, for $k-2\GEQ l\GEQ 1$, a node
$(a_l,b_l)$ in path $\rho_l$ with $a_l\LEQ a_{l+1}$ and $b_l>b_{l+1}$.
Now, from the way they are located, the $k+1$ nodes $(a_{k+1},b_{k+1})$, \ldots, $(a_1,b_1)$
cannot belong to $k$ or fewer paths
but clearly belong to the $k$-path $\Pi$.
We have thus reached the desired final contradiction and completed the
proof.
\end{proof}

\begin{example}
\label{ex:2.15z}
\rm
This is an example of a 7-path which is not ordered and equivalent ordered $k$-paths.
Let $\Pi$ be the 7-path
\\
$\begin{array}{*{6}{l}}
( & ((1,1),(4,2),(6,4)), &
 ((1,2),(3,3),(4,3),(6,5)), &
 ((1,3),(3,4),(4,4)),
\\
 & ((3,5),(4,5)), &
 ((1,4),(2,4),(5,4),(6,6)), &
 ((1,5),(3,6),(4,7)),
\\
 & ((2,1),(5,3)) \quad ),
\end{array}$
\\
let $\Pi''$ be the 5-path
\\
$\begin{array}{*{6}{l}}
( & ((1,1),(2,1),(4,2)), &
 ((1,2),(3,3),(4,3),(5,3),(6,4)),
\\
 & ((1,3),(3,4),(4,4),(5,4),(6,5)), &
 ((1,4),(2,4),(3,5),(4,5),(6,6)),
\\
 & ((1,5),(3,6),(4,7)) \quad ),
\end{array}$
\\
and let $\Pi'''$ be the 7-path
\\
$\begin{array}{*{6}{l}}
( & ((2,1),(5,3)), &
 ((1,1),(4,2),(6,4)), &
 ((1,2),(3,3),(4,3),(6,5)),
\\
 & ((5,4),(6,6)), &
 ((1,3),(3,4),(4,4)), &
 ((1,4),(2,4),(3,5),(4,5)),
\\
 & ((1,5),(3,6),(4,7)) \quad ).
\end{array}$
\par
These three $k$-paths are described diagrammatically in
Table~\ref{tabl:1}, where the nodes on each path of a $k$-path are
represented by the index of that path.
$\Pi''$ is the equivalent ordered $k'$-path (here $k'=5$) produced by the algorithm of
Theorem~\ref{thm:2.14a}.
There are other equivalent ordered $k$-paths.
$\Pi'''$ is one such $k$-path with $k=7$.
%
\setcounter{table}{0}
\begin{table}[h]
\begin{center}
$\begin{array}{*{7}{c}}
\\[-7ex]
\setlength{\arraycolsep}{6pt}
\begin{array}{*{7}{c}}
\clx 1 & \clx 2 & \clx 3 & \clx 5 & \clx 6 &   &   \\
\clx 7 &   &   & \clx 5 &   &   &   \\
  &   & \clx 2 & \clx 3 & \clx 4 & \clx 6 &   \\
  & \clx 1 & \clx 2 & \clx 3 & \clx 4 &   & \clx 6 \\
  &   & \clx 7 & \clx 5 &   &   &   \\
  &   &   & \clx 1 & \clx 2 & \clx 5 &   \\
\end{array}
&\hspace{10pt}&
\setlength{\arraycolsep}{6pt}
\begin{array}{*{7}{c}}
\clx 1 & \clx 2 & \clx 3 & \clx 4 & \clx 5 &   &   \\
\clx 1 &   &   & \clx 4 &   &   &   \\
  &   & \clx 2 & \clx 3 & \clx 4 & \clx 5 &   \\
  & \clx 1 & \clx 2 & \clx 3 & \clx 4 &   & \clx 5 \\
  &   & \clx 2 & \clx 3 &   &   &   \\
  &   &   & \clx 2 & \clx 3 & \clx 4 &   \\
\end{array}
&\hspace{10pt}&
\setlength{\arraycolsep}{6pt}
\begin{array}{*{7}{c}}
\clx 2 & \clx 3 & \clx 5 & \clx 6 & \clx 7 &   &   \\
\clx 1 &   &   & \clx 6 &   &   &   \\
  &   & \clx 3 & \clx 5 & \clx 6 & \clx 7 &   \\
  & \clx 2 & \clx 3 & \clx 5 & \clx 6 &   & \clx 7 \\
  &   & \clx 1 & \clx 4 &   &   &   \\
  &   &   & \clx 2 & \clx 3 & \clx 4 &   \\
\end{array}
\\[6ex]
\Pi && \Pi'' && \Pi'''
\\[-2ex]
\end{array}
$
\end{center}
\caption{%
constituent paths indicated by path indices and colours%
\label{tabl:1}
}
\end{table}
\end{example}

Finally for this section we establish some results which will play some part in the arguments in Section~\ref{sec:3a}.
\begin{lemma}
\label{lem:3.3a}
Let $\Pi=(\pi_1,\ldots,\pi_k)$ be an ordered $k$-path in a diagram
$D$, and let $(a',b')$ be a node of $D$ which is not in $s(\Pi)$.
\par
\begin{tabular}{rp{5.7in}}
(i) &
If there is a path $\pi_j$, $1\LEQ j\LEQ k$, with a pair of nodes
$(a_1,b')$ and $(a_2,b')$ with $a_1<a'<a_2$, then $\Pi$ may be extended
to an ordered $k$-path
$\Pi'=(\pi_1,\ldots,\pi_{j-1},\pi_j',\pi_{j+1},\ldots,\pi_k)$
where $s(\pi_j')=s(\pi_j)\cup\{(a',b')\}$.
\\[2ex]
(ii) &
If there is  no path $\pi_j$, $1\LEQ j\LEQ k$, with a pair of nodes
$(a_1,b')$ and $(a_2,b')$ with $a_1<a'<a_2$, then $\Pi$ may be extended
to an ordered $(k+1)$-path
$\Pi'=(\pi_1,\ldots,\pi_{k'},\pi',\pi_{k'+1},\ldots,\pi_k)$
for some $k'$, $0\LEQ k'\LEQ k$, where $s(\pi')=\{(a',b')\}$.
\end{tabular}
\end{lemma}
\begin{proof} (i)
In this case, it is immediate that $\pi_j'$ is a path in $D$.
If $i<j$ and $(a,b)\in s(\pi_{i})$ satisfies $a\LEQ a'$,
then $a<a_2$.
So, $b<b'$.
Hence, $s(\pi_{i})\prec s(\pi_j')$.
However, if $j<i$ and $(a,b)\in s(\pi_{i})$ satisfies $a'\LEQ a$,
then $a_1<a$.
So, $b'<b$.
Hence, $s(\pi_j')\prec s(\pi_{i})$.
It follows that
$\Pi'=(\pi_1,\ldots,\pi_{j-1},\pi_j',\pi_{j+1},\ldots,\pi_k)$
is an ordered $k$-path.
\par
(ii)
Let $\pi'$ be the path with $s(\pi')=\{(a',b')\}$.
Let $l$ be the maximum index, $0\LEQ l\LEQ k$, such that
$s(\pi_j)\prec s(\pi')$ for all $j$ with $1\LEQ j\LEQ l$;
$l=0$ indicates $s(\pi_1)\not\prec s(\pi')$.
If $l=k$ then we may take
$\Pi'=(\pi_1,\ldots,\pi_k,\pi')$.
If $l<k$ then $s(\pi_{l+1})\not\prec s(\pi')$.
By Lemma~\ref{lem:2.12a}, we get
$s(\pi')\prec s(\pi_j)$ for $l+2\LEQ j\LEQ k$.
Also, by Lemma~\ref{lem:2.11a}, we get  $s(\pi')\prec s(\pi_{l+1})$.
So we may take
$\Pi'=(\pi_1,\ldots,\pi_l,\pi',\pi_{l+1},\ldots,\pi_k)$.
\end{proof}
\begin{corollary}
\label{cor:3.4a}
Let $\Pi=(\pi_1,\ldots,\pi_k)$ be an ordered $k$-path in a diagram
$D$, and let $(a_i',b_i')$, $1\LEQ i\LEQ l$, be $l$ distinct nodes of
$D$ which are not in $\Pi$.
If no path $\pi_j$, $1\LEQ j\LEQ k$, contains a pair of nodes of the
form $(a_{i,j,1},b_i')$, $(a_{i,j,2},b_i')$ with
$a_{i,j,1}<a_i'<a_{i,j,2}$ for any $i$ satisfying $1\LEQ i\LEQ l$,
then the paths $((a_i',b_i'))$ may be inserted into the sequence $\Pi$
to give an ordered $(k+l)$-path.
\end{corollary}
\begin{proof}
The proof is a repeated application of Lemma~\ref{lem:3.3a}(ii).
\end{proof}
\par
We will refer to the process described in Corollary~\ref{cor:3.4a}
as \emph{extending an ordered $k$-path by paths of length one}.

\begin{proposition}\label{propOrdSpecial}
Let $\lambda\vDash n$ and let $y\in Y(\lambda)$.
Set $D=D(y,\lambda)$.
Suppose that $D=s(\Pi)$ for some ordered $k$-path $\Pi$.
Then (i) and (ii) below hold.

(i) If $D(\Pi)$ is admissible and $D(\Pi)=D(w_{D(\Pi)},\lambda)$, then $D=D(\Pi)$.

(ii) If $\Pi$ has type $\lambda'$, then $D=D(\Pi)$ and $D$ is special.
\end{proposition}
\begin{proof}
(i) This is immediate from Lemmas~\ref{BlemmaSpecial} and~\ref{lBasic}.

(ii) If $\Pi$ has type $\lambda'$, then $k$ equals  the number of parts of $\lambda'$ and, moreover,  $D(\Pi)\in \mathcal D^{(\lambda)}$ is a special diagram.
(In fact, $D(\Pi)$ is the unique element of $\mathcal D^{(\lambda,\mu)}$ where $\mu=(\mu_1,\ldots,\mu_k)\vDash n$ with $\mu_i=$ length of $\pi_i$ for $1\LEQ i\LEQ k$.)
Combining item (i) of this proposition with Remark~\ref{rSpecial} and Result~\ref{res:8a}(ii) we get $D=D(\Pi)$ and hence the desired result.
\end{proof}

\end{section}
\begin{section}
{Determining the rim for certain families of cells}
\label{sec:3a}
For an arbitrary composition $\lambda=(\lambda_1,\ldots,\lambda_r)\vDash n$,
let $\lambda_{*}=(\lambda_1,\ldots,\lambda_r,1)\vDash n+1$.
In \cite[Section~4]{MPa17}, there is a well-defined mapping $\psi$ from
the set of admissible diagrams in $\mathcal{D}^{(\lambda)}$ to the set
of admissible diagrams in $\mathcal{D}^{(\lambda_*)}$, which induces an
injective mapping $\theta_{*}\colon Y(\lambda)\to Y(\lambda_{*})$.
For a given admissible diagram $D$ in $\mathcal{D}^{(\lambda)}$, the
diagram $D\psi$ is obtained by examining all diagrams constructed from
$D$ by appending an $(r+1)$-th row with a single node to $D$ and
selecting the diagram which is admissible and such that the column index
of the new node is minimal.
The mapping $\psi$ induces an injection $\mathcal E^{(\lambda)}\to\mathcal E^{(\lambda_*)}$ and
the mapping $\theta_{*}$ is then given by $w_{D}\mapsto w_{D\psi}$.
In  Proposition~\ref{prop:3.2a}, we obtain a condition on $\lambda$ which
ensures that the mapping $\theta_{*}$ is a bijection.
\begin{lemma}
\label{lem:3.1a}
Let $r\GEQ 2$, let $n\GEQ 2$,
let $\lambda=(\lambda_1,\ldots,\lambda_r)\vDash n$
be an $r$-part composition with $\lambda_r=1$.
Also let $\tilde\lambda=(\lambda_1,\ldots,\lambda_{r-1})\vDash n-1$ and 
let $\lambda'$, the conjugate of $\lambda$, be given by $\lambda'=(\lambda_1',\ldots,\lambda_{r'}')$.
\par
Suppose $D$ is an admissible diagram in $\mathcal{D}^{(\lambda)}$.
Then the diagram $\tilde D$, obtained from $D$ by removing the
$r$-th row, is an admissible diagram in $\mathcal{D}^{(\tilde\lambda)}$.
Moreover, for $1\LEQ k\LEQ r'$, every $k$-path in $D$ of
length $\lambda_1'+\cdots+\lambda_k'$ contains the node on the
$r$-th row of $D$.
\end{lemma}
\begin{proof}
By definition, $D$ has subsequence type $\lambda'$.
Let $\nu$ be the subsequence type of $\tilde D$.
By Result~\ref{res:8a}(i), $\nu\unlhd(\tilde\lambda)'$
$=(\lambda_1'-1,\lambda_2',\ldots,\lambda_{r'}')$.
Let $N$ denote the node on the $r$-th row of $D$,
let $\Pi$ be a $k$-path in $D$ of length
$\lambda_1'+\cdots+\lambda_k'$, where $1\LEQ k\LEQ r'$.
If $\Pi$ did not contain $N$, then $\Pi$ would also be a $k$-path in
$\tilde D$.
However, $k$-paths in $\tilde D$ have length at most
$\lambda_1'+\cdots+\lambda_k'-1$.
Hence, $\Pi$ must contain $N$.
(Alternatively, this follows from Result~\ref{res:6a}.)
Moreover, $s(\Pi)-\{N\}$ is the support of a $k$-path in
$\tilde D$ of length $\lambda_1'+\cdots+\lambda_k'-1$.
Thus $\nu=(\tilde\lambda)'$ and so $\tilde D$ is admissible.
\end{proof}
\begin{proposition}
\label{prop:3.2a}
Let $r\GEQ 2$, let $n\GEQ 2$,
let $\lambda=(\lambda_1,\ldots,\lambda_r)\vDash n$
be an $r$-part composition with $\lambda_r=1$.
Then the mapping $\theta_{*}\colon Y(\lambda)\to Y(\lambda_{*})$,
described in \cite[Section~4]{MPa17}, is a bijection.
\end{proposition}
\begin{proof}
Since $\theta_{*}$ is injective by \cite[Theorem~4.3]{MPa17}, we need
only prove that it is surjective.
Let $y\in Y(\lambda_{*})$ and let $D_{*}=D(y,\lambda_{*})$.
So $y=w_{D_{*}}$ and $D_{*}$ is admissible.
Thus, writing $\lambda'=(\lambda_1',\ldots,\lambda_{r'}')$, we see that $D_{*}$ has subsequence type
$(\lambda_{*})'=(\lambda_1'+1,\lambda_2',\ldots,\lambda_{r'}')$.
Let $N=(r,a)$ and $N_{*}=(r+1,b)$ be the nodes on the $r$-th and
$(r+1)$-th rows of $D_{*}$, and let $D=D_{*}-\{N_{*}\}$.
By Lemma~\ref{lem:3.1a}, $D$ is an admissible diagram in
$\mathcal{D}^{(\lambda)}$ and
if $\Pi_{*}$ is a $k$-path in $D_{*}$ of
length $\lambda_1'+\cdots+\lambda_k'+1$, then $\Pi_{*}$  contains $N_{*}$,
$1\LEQ k\LEQ r'$.
Again by Lemma~\ref{lem:3.1a},  $s(\Pi_{*})-\{N_{*}\}$ is the support of a $k$-path in
$D$ which contains $N$.
In particular, as $D_{*}$ has a 1-path of length $\lambda_1'+1$
containing both $N$ and $N_{*}$, $a\LEQ b$.
\par
We now construct a diagram $\bar D$ from $D$ by adding the node
$\bar N=(r+1,a)$ as the single node on the $(r+1)$-th row.
Since every $k$-path of length $\lambda_1'+\cdots+\lambda_k'$ in $D$
contains $N$, $1\LEQ k\LEQ r'$, each may be extended to a $k$-path
of length $\lambda_1'+\cdots+\lambda_k'+1$ in $\bar D$ by adding the
node $\bar N$.
Hence, the subsequence type $\nu$ of $\bar D$ satisfies
$(\lambda_{*})'\unlhd\nu$.
By Result~\ref{res:8a}(i), $\nu\unlhd(\lambda_{*})'$.
Hence, $\nu=(\lambda_{*})'$ and $\bar D$ is admissible.
From \cite[Section~4]{MPa17}, or equivalently from the paragraph
preceding Lemma~\ref{lem:3.1a}, $\bar D=D\psi$.
\par
If $a<b$ then $w_{\bar D}\neq w_{D_{*}}$ since no column of $D$ is
empty.
Since $t^{\bar D}w_{D_{*}}$ is a standard $\bar D$-tableaux and
$\bar D$ is admissible, it follows from Result~\ref{res:3a}
that $w_{D_{*}}\notin Y(\lambda_{*})$.
Since this is contrary to hypothesis, $a=b$, $D_{*}=\bar D=D\psi$.
So $y=w_{D_{*}}=w_D\theta_{*}$.
This concludes the proof.
\end{proof}
\par
Recall that the rim $Y(\lambda)$ of the right cell $\mathfrak C(\lambda)$ is given by $Y(\lambda)=\{w_D\colon D\in\mathcal E^{(\lambda)}\}$.
Thus, informally, we see that the elements of $Y(\lambda_{*})$, with
$\lambda$ as in Proposition~\ref{prop:3.2a}, are obtained from the elements
of $Y(\lambda)$ by constructing the diagrams in $\mathcal E^{(\lambda)}$, then forming the diagrams in $\mathcal E^{(\lambda_*)}$
by appending to each diagram in $\mathcal E^{(\lambda)}$ a new node in the column of the node on the
last row, and taking the corresponding `$w$' of the new diagrams.
In~\cite[Remark~4.4]{MPa17} it is described how this process  relates to the induction of cells (see~\cite[Proposition~3.15]{BVo83}).

We turn to deal with some special compositions and we begin with
the case of compositions in which at most the first
two parts are greater than 1.
\begin{theorem}
\label{thm:3.5a}
Let $r\GEQ3$ and $s\GEQ t\GEQ 1$.
Let $\lambda=(\lambda_1,\ldots,\lambda_r)$ be a composition where
$(\lambda_1,\lambda_2)$ is a permutation of $(s,t)$
and $\lambda_i=1$ if $i>2$.
\par
\begin{tabular}{rp{5.5in}}
(i) &
If $(\lambda_1,\lambda_2)=(s,t)$, then
$\mathcal E^{(\lambda)}=\mathcal E_s^{(\lambda)}=\{V(\lambda)\}$. 
\\[1ex]
(ii)  &
If $(\lambda_1,\lambda_2)=(t,s)$, then
$\mathcal E^{(\lambda)} =\{{D_{t,s,u}}\colon 1\LEQ u\LEQ s-t+1\}$,
where,
$D_{t,s,u}=\{(1,u)\}\cup
\{(1,i)\colon s-t+2\LEQ i\LEQ s\}\cup
\{(2,i)\colon 1\LEQ i\LEQ s\}\cup
\{(i,u)\colon 3\LEQ i\LEQ r\}$.
Hence, $\mathcal E^{(\lambda)}=\mathcal E_s^{(\lambda)}$.
Moreover, $|\mathcal E^{(\lambda)}|=s-t+1$.
\end{tabular}
\end{theorem}
\begin{proof}
If $(\lambda_1,\lambda_2)=(s,t)$ then $\lambda$ is a partition and this case is covered in Remark~\ref{rem:2.2}.

Now let $s\GEQ t\GEQ 1$ and suppose that $\lambda=(t,s,1)$.
Then $\lambda'$ is the partition $3^12^{t-1}1^{s-t}$.
Let $d\in Y(\lambda)$ and let $D=D(d,\lambda)$ so that $d=w_D$.
Since $d\in Z(\lambda)$, $D$ is an admissible diagram.
It follows that $D$ contains a $t$-path $\Pi'$ of length $2t+1$.
We can assume that $\Pi'$ is ordered in view of Theorem~\ref{thm:2.14a}.

Since $D$ has exactly 3 rows, all constituent paths of $\Pi'$ have length $\LEQ3$.
Moreover, at least one of the constituent paths of $\Pi'$ has length 3 (otherwise the length of $\Pi'$ would be $\LEQ2t$).
Since $D$ has exactly one node in the third row, exactly one of the constituent paths of $\Pi'$ has length 3 (alternatively this can be deduced from the fact that $D$ is of subsequence type $\lambda'$).
Let $\pi$ be the unique path in $\Pi'$ of length 3.
Then $\pi$  contains one node in each one of the three rows of $D$.
It follows that every one of the remaining $t-1$ paths of $\Pi'$ has length 2 and contains one node in each of the first two rows of $D$.

The remaining $s-t$ nodes of $D$ (the size of $D$ is $s+t+1$) all belong to the second row of~$D$.
Clearly the nodes of $\pi$ which are located in the first and third rows of $D$ cannot both have column index which equals the column index of any of these $s-t$ nodes.
Hence, by Corollary~\ref{cor:3.4a}, we may extend $\Pi'$
to an ordered $s$-path $\Pi$ by
the $s-t$ paths of length 1 whose nodes are
the remaining nodes on the second row of $D$.
Clearly $\Pi$ has type $\lambda'$ and the support of $\Pi$ is the diagram $D$.
By Proposition~\ref{propOrdSpecial}(ii), $D$ is special.
In particular, since $D$ is a rearrangement of a Young diagram, every column of $D$ contains a node located in the second row of $D$ (so the nodes in the second row of $D$ are the nodes $(2,r)$ for $1\LEQ r\LEQ s$).

Suppose that the first node on the first row of $D\,(=D(d,\lambda))$ is $(1,u)$.
Since the first row of $D$  has $t$ nodes, $1\LEQ u\LEQ s-t+1$.
Form a diagram $F$ whose first row nodes are $(1,u)$, $(1,s-t+2)$,
\ldots, $(1,s)$, whose second row nodes are the same as $D$, and
whose single third row node is $(3,u)$.
Then $F\in\mathcal{D}^{(\lambda)}$, $F$ is special, hence admissible, and
$t^Fw_D$ is a standard $F$-tableau.
By Corollary~\ref{lemmaSpecial},  $D=F$.
We may refer to the admissible diagram $F$ just constructed as $F_u$.
It is also clear the $w_{F_u}$, for $1\LEQ u\LEQ s-t+1$ are mutually
non-prefixes of one another.
It follows that $\mathcal E^{(\lambda)}=\mathcal E_s^{(\lambda)}=\{{F_u}\colon 1\LEQ u\LEQ s-t+1\}$.
In particular $|\mathcal E^{(\lambda)}|=s-t+1$.

To complete the proof for $r>3$, we have only to use
Proposition~\ref{prop:3.2a} and the remarks following its proof.
\end{proof}

\begin{remark}
Combining Theorem~\ref{thm:3.5a}  with Remark~\ref{res:10a} we can determine $\mathcal E^{(\lambda)}$ for $\lambda=(1^r,a,b)$ where $r,a,b\GEQ 1$.
\end{remark}

Before we consider all compositions with three parts,
we describe some special (and hence admissible) diagrams which we will
use in Theorem~\ref{thm:3.7a}.
\begin{example}
\label{ex:3.6a}
Let $s\GEQ t\GEQ u\GEQ 1$.
\par
(i): 
If $\lambda=(s,u,t)$ and $C\subseteq\{1,\ldots,t\}$ with $|C|=u$,
then $F_{C}=\{(1,i)\colon 1\LEQ i\LEQ s\}\cup
\{(2,i)\colon i\in C\}\cup\{(3,i)\colon 1\LEQ i\LEQ t\}$ is a special
diagram. 
\par
If $\lambda=(8,3,5)$ and $C=\{2,3,4\}$, then $F_{C}$ is the diagram
$\setlength{\arraycolsep}{3pt}
\begin{array}{*{8}{c}}
\times&\times&\times&\times&\times&\times&\times&\times\\
      &\times&\times&\times&      &      &      &      \\
\times&\times&\times&\times&\times&      &      &      \\
\end{array}$
\par
(ii): 
If $\lambda=(t,s,u)$ and $C\subseteq\{1,\ldots,s-t+u\}$ with
$|C|=u$, then $G_{C}=\{(1,i)\colon i\in (C\cup\{s-t+u+1,\ldots,s\})\}\cup
\{(2,i)\colon 1\LEQ i\LEQ s\}\cup\{(3,i)\colon i\in C\}$ is a special
 diagram.
\par
If $\lambda=(5,8,3)$ and $C=\{2,4,5\}$, then $G_{C}$ is the diagram
$
\setlength{\arraycolsep}{3pt}
\begin{array}{*{8}{c}}
      &\times&      &\times&\times&      &\times&\times\\
\times&\times&\times&\times&\times&\times&\times&\times\\
      &\times&      &\times&\times&      &      &      \\
\end{array}$
\par
(iii): 
If $\lambda=(t,u,s)$ and $C\subseteq\{s-t+1,\ldots,s\}$ with
$|C|=u$, then $H_{C}=\{(1,i)\colon s-t+1\LEQ i\LEQ s\}\cup
\{(2,i)\colon i\in C\}\cup\{(3,i)\colon 1\LEQ i\LEQ s\}$ is a special
 diagram.
\par
If $\lambda=(5,3,8)$ and $C=\{5,6,8\}$, then $H_{C}$ is the diagram
$\setlength{\arraycolsep}{3pt}
\begin{array}{*{8}{c}}
      &      &      &\times&\times&\times&\times&\times\\
      &      &      &      &\times&\times&      &\times\\
\times&\times&\times&\times&\times&\times&\times&\times\\
\end{array}$
\par
(iv): 
If $\lambda=(u,s,t)$ and $C\subseteq\{t-u+1,\ldots,s\}$ with $|C|=u$,
then $K_{C}=\{(1,i)\colon i\in C\}\cup\{(2,i)\colon 1\LEQ i\LEQ s\}
\cup\{(3,i)\colon i\in(\{1,\ldots,t-u\}\cup C)\}$ is a special
 diagram.
\par
If $\lambda=(3,8,5)$ and $C=\{3,5,7\}$, then $K_{C}$ is the diagram
$\setlength{\arraycolsep}{3pt}
\begin{array}{*{8}{c}}
      &      &\times&      &\times&      &\times&      \\
\times&\times&\times&\times&\times&\times&\times&\times\\
\times&\times&\times&      &\times&      &\times&      \\
\end{array}$
\par
(v): 
If $\lambda=(u,t,s)$ then $L=\{(1,i)\colon s-u+1\LEQ i\LEQ s\}
\cup\{(2,i)\colon s-t+1\LEQ i\LEQ s\}\cup\{(3,i)\colon 1\LEQ i\LEQ s\}$ is
a special  diagram.
\par
If $\lambda=(3,5,8)$, then $L$ is the diagram
$\setlength{\arraycolsep}{3pt}
\begin{array}{*{8}{c}}
      &      &      &      &      &\times&\times&\times\\
      &      &      &\times&\times&\times&\times&\times\\
\times&\times&\times&\times&\times&\times&\times&\times\\
\end{array}$
\end{example}
\begin{theorem}
\label{thm:3.7a}
Let $s\GEQ t\GEQ u\GEQ 1$ and let
$\lambda=(\lambda_1,\lambda_2,\lambda_3)$ be a composition which is a
permutation of $(s,t,u)$.
Then, with $|C|=u$ and  diagrams $G_C$ and $H_C$ as described in Example~\ref{ex:3.6a}, \\[1ex]
$\begin{array}{lll}
\mathcal E^{(\lambda)} & = & 
\left\{
\begin{array}{lll}
\{{V(\lambda)}\}
&\mbox{ if }&\lambda=(s,t,u),
\\[1ex]
\{{G_{C}}\colon C\subseteq\{1,\ldots,s-t+u\}\}
&\mbox{ if }&\lambda=(t,s,u),
\\[1ex]
\{{H_{C}}\colon C\subseteq\{s-t+1,\ldots,s\}\}
&\mbox{ if }&\lambda=(t,u,s),
\\[1ex]
\{\dot M\colon M\in\mathcal E^{(\dot\lambda)}\}
&\mbox{ if }&\lambda=(u,t,s),\ (u,s,t)\mbox{ or }(s,u,t).
\end{array}
\right.
\end{array}$
\\[1ex]
Moreover, in all cases, $\mathcal E^{(\lambda)}=\mathcal E_s^{(\lambda)}$ and
the value of $|\mathcal E^{(\lambda)}|$ is given in Table~\ref{tbl:2a}.
\end{theorem}
\begin{proof}
Choose distinct $i_1,i_2,i_3\in\{1,2,3\}$ so that $\lambda_{i_1}=s$,
$\lambda_{i_2}=t$ and $\lambda_{i_3}=u$.
Also let $\lambda$ be as in the statement of the theorem.
\par
Let $d\in Y(\lambda)$ and let $D=D(d,\lambda)$.
Then $D$ is an admissible diagram; so $D$ has subsequence type
$\lambda'$.
%
%
%
Since $\lambda'=3^u2^{t-u}1^{s-t}$ and
$D$ is admissible, it has a $t$-path
$\Pi'=(\pi_1',\ldots,\pi_t')$ of length $2t+u$, which we may assume
to be ordered by Theorem~\ref{thm:2.14a}.
For $1\LEQ i\LEQ 3$, let $z_i'$ be the number of paths of length $i$
in $\Pi'$.
Since each path of length 3 has a node on each row, $z_3'\LEQ u$.
Simple counting gives $z_1'+z_2'+z_3'=t$ and $z_1'+2z_2'+3z_3'=2t+u$.
Hence, $z_2'+2z_3'=t+u$.
Also, $z_2'+z_3'\LEQ t$.
Hence, $z_3'\GEQ u$.
So, $z_3'=u$, $z_1'=0$ and $z_2'=t-u$.
\par
The $u$ paths of $\Pi'$ of length 3 together contain all nodes on
row $i_3$.
The remaining $t-u$ paths of $\Pi'$ have length 2 but have no nodes
on row $i_3$. Hence, these paths contain all remaining $t-u$ nodes
on row $i_2$.
\par
The $s-t$ remaining nodes of $D$ are on row $i_1$.
None of these may be used to extend a path of length 2 in $\Pi'$ to
a path of length 3; otherwise, $D$ would have a $t$-path of length
$\GEQ 2t+u+1$.
Moreover, if $i_1=2$ and $\pi$ is a path of length 3 in $\Pi$, it is clear that the nodes of $\pi$ in rows $i_2$ and $i_3$ cannot both have column index which equals the column index of any of these $s-t$ nodes of $D$ on row $i_1$.
Hence, by Corollary~\ref{cor:3.4a}, we may extend the ordered $t$-path
$\Pi'$ by the $s-t$ paths of length one, whose nodes are these
remaining nodes on row $i_1$, to an ordered $s$-path
$\Pi=(\pi_1,\ldots,\pi_s)$ of length $s+t+u$.
Clearly, $\Pi$ has $z_i$ paths of length $i$ where $z_1=s-t$,
$z_2=t-u$, and $z_3=u$.
Since the support of $\Pi$ is the whole of $D$ and $\Pi$ has type $\lambda'$, $D$ is special by Proposition~\ref{propOrdSpecial}(ii).
So, $\mathcal E^{(\lambda)}=\mathcal E_s^{(\lambda)}$.
\par
Let $A$ and $B$ be the sets of columns in $D$ containing nodes on rows
$i_2$ and $i_3$, respectively.
Since $D$ is special, the nodes on row $i_1$ are in columns 1, \ldots, $s$ and, in addition, $B\subseteq A\subseteq \{1,\ldots,s\}$.

Consider first the case $\lambda=(t,s,u)$.
Let $C$ be the first $u$ columns in $A$.
Since $|A-C|=t-u$, $C\subseteq\{1,\ldots,s-t+u\}$.
Then $t^{G_{C}}w_D$, which is obtained from $t_D$ by moving the
entries on the third row to the left into the corresponding positions
in $G_{C}$ and moving the last $t-u$ entries on the first row
to the right into the last $t-u$ columns, is a standard
$G_{C}$-tableau.
Since $G_{C}$ is special, $D=G_C$ by Corollary~\ref{lemmaSpecial}.

Every $u$-subset $C$ of $\{1,\ldots,s-t+u\}$ gives rise to an admissible
 diagram $G_{C}\in \mathcal E_s^{(\lambda)}$.
Moreover, if $C_1$ and $C_2$ are distinct $u$-subsets of
$\{1,\ldots,s-t+u\}$, then it is immediate that
$t^{G_{C_1}}w_{G_{C_2}}$ is not a standard tableau.
Hence,
$\mathcal E^{(\lambda)}=\mathcal E_s^{(\lambda)}=
\{ {G_{C}}\colon
C\mbox{ a } u\mbox{-subset of }\{1,\ldots,s-t+u\}\}$
and $|\mathcal E^{(\lambda)}|=\dbinom{s-t+u}{u}$.
\par
%
%
Now consider $\lambda=(t,u,s)$.
Let $\tilde A=\{s-t+1,\ldots,s\}$ and let $C$ be the subset of
$\tilde A$ whose elements occupy the same positions in $\tilde A$ as
those of $B$ occupy in $A$.
Then $t^{H_{C}}w_D$, which is obtained from $t_D$ by moving the
entries on the first row to the right into the columns given by
$\tilde A$ and the entries in the second row to the columns given by
$C$, is a standard $H_{C}$-tableau.
By Corollary~\ref{lemmaSpecial}, $D=H_{C}$ since $H_C$ is special.
\par
Every $u$-subset $C$ of $\{s-t+1,\ldots,s\}$ gives rise to an admissible
 diagram $H_{C}\in \mathcal E_s^{(\lambda)}$.
Moreover, if $C_1$ and $C_2$ are distinct $u$-subsets of
$\{s-t+1,\ldots,s\}$, then it is immediate that
$t^{H_{C_1}}w_{H_{C_2}}$ is not a standard tableau.
Hence,
$\mathcal E^{(\lambda)}=\mathcal E_s^{(\lambda)}=
\{ {H_{C}}\colon
C\mbox{ a } u\mbox{-subset of }\{s-t+1,\ldots,s\}\}$
and $|\mathcal E^{(\lambda)}|=\dbinom{t}{u}$.
\par
%
%

Next consider $\lambda=(s,t,u)$.
Then $\lambda$ is a partition and this case is covered in Remark~\ref{rem:2.2}.

Finally, the  diagrams in $\mathcal E^{(\lambda)}$
for $\lambda=(u,s,t)$, $(s,u,t)$ and $(u,t,s)$ are obtained by rotating
through $180^{\circ}$ those in $\mathcal E^{(\mu)}$ for $\mu=(t,s,u)$,
$(t,u,s)$ and $(s,t,u)$, respectively (see Remark~\ref{res:10a}).
Since $\mathcal E^{(\lambda)}=\mathcal E_s^{(\lambda)}$ for these cases also,
$|\mathcal E^{(\lambda)}|= \dbinom{s-t+u}{u}$ or $\dbinom{t}{u}$ or $1$
according as $\lambda=(u,s,t)$ or $(s,u,t)$ or $(u,t,s)$.
This completes the proof.
\end{proof}
\par
Direct arguments can also be given for the cases of $\lambda=(s,u,t)$,
$(u,s,t)$ and $(u,t,s)$. The diagrams in $\mathcal E^{(\lambda)}$ arising in these cases are described in
Examples~\ref{ex:3.6a} (i), (iv) and (v).
\begin{table}[h]
\centering
$\begin{array}{ccccccccccccc}
\lambda & (s,t,u) & (s,u,t) & (t,s,u) & (t,u,s) & (u,s,t) & (u,t,s)
\\\hline
\\[-1ex]
|\mathcal E^{(\lambda)}| & 1 & \dbinom{t}{u} & \dbinom{s-t+u}{u} & \dbinom{t}{u}
& \dbinom{s-t+u}{u} & 1
\end{array}$
\caption{
$s\GEQ t\GEQ u\GEQ 1$. (Theorem~\ref{thm:3.7a})
}
\label{tbl:2a}
\end{table}

We conclude this section with a theorem in which
$\mathcal E^{(\lambda)}$ and $\mathcal E_s^{(\lambda)}$ are determined for the family
of compositions $\lambda=(1,2^{r-2},1)$ and, as it turns out, $\mathcal E^{(\lambda)}\ne\mathcal E_s^{(\lambda)}$ for $r>3$.
We begin by identifying certain admissible diagrams for such
compositions
\begin{example}
\label{ex:3.14a}
Let $r\GEQ3$ and let $\lambda=(1,2^{r-2},1)$.
Define $P^{(0)}=\{(i,1)\colon 1\LEQ i\LEQ r\}\cup
\{(i,2)\colon 2\LEQ i\LEQ r-1\}$, 
$P^{(r-2)}=\{(i,1)\colon 2\LEQ i\LEQ r-1\}\cup\{(i,2)\colon 1\LEQ i\LEQ r\}$ and, for $1\LEQ v\LEQ r-3$,
define
$P^{(v)}=\{(i,1)\colon 2\LEQ i\LEQ v+1\}\cup
\{(i,2)\colon 1\LEQ i\LEQ r\}\cup\{(i,3)\colon v+2\LEQ i\LEQ r-1\}$.
Thus,
\begin{center}
$
\setlength{\arraycolsep}{0pt}
{\small
\begin{array}{cccc}
P^{(0)} =
\begin{array}{cc}
 \times &        \\
 \times & \times \\
 \vdots & \vdots \\
 \times & \times \\
 \times & \times \\
 \vdots & \vdots \\
 \times & \times \\
 \times &        \\
 \end{array},\quad
&
P^{(r-2)}=
\begin{array}{cc}
& \times\\
\times & \times \\
 \vdots & \vdots \\
 \times & \times \\
 \times & \times \\
 \vdots & \vdots \\
 \times & \times \\
& \times
\end{array},\quad
&
P^{(v)} =
\begin{array}{ccc}
        & \times &  \\
 \times & \times &  \\
 \vdots & \vdots &  \\
 \times & \times &  \\
        & \times & \times \\
        & \vdots & \vdots \\
        & \times & \times \\
        & \times &  \\
 \end{array}
&
\makebox[1.5cm]{\hspace{\fill} for }1\LEQ v\LEQ r-3.
\end{array}
}
$
\end{center}
Since $\lambda'=(r,r-2)$ and clearly each of these diagrams $P^{(v)}$, $0\LEQ v\LEQ r-2$, has a
path of length $r$ and a 2-path containing all $2r-2$ nodes,
they are all admissible.
Moreover,  $P^{(v)}=D(w_{P^{(v)}},\lambda)$ for $0\LEQ v\LEQ r-2$.
The tableaux $t_v=t_{P^{(v)}}$ are given by
{\footnotesize
\begin{center}
$
\setlength{\arraycolsep}{3pt}
\begin{array}{cccc}
t_{0} =
\begin{array}{cc}
 1          &                \\
 2          & r\!+\!1        \\
 \vdots     & \vdots         \\
 i          & r\!+\!i\!-\!1  \\
 i\!+\!1    & r\!+\!i        \\
 \vdots     & \vdots         \\
 r\!-\!1    & 2r\!-\!2       \\
 r          &                \\
 \end{array},\quad
&
t_{(r-2)}=\begin{array}{cc}
& r-1\\
1 & r\\
\vdots & \vdots\\
i & r+i-1\\
i+1 & r+i\\
\vdots & \vdots\\
r-2 & 2r-3\\
& 2r-2
\end{array}
,\quad
&
t_v =
\begin{array}{ccc}
        & v\!+\!1        &                \\
 1      & v\!+\!2        &                \\
 \vdots & \vdots         &                \\
 v      & 2v\!+\!1       &                \\
        & 2v\!+\!2       & v\!+\!r\!+\!1  \\
        & \vdots         & \vdots         \\
        & v\!+\!r\!-\!1  & 2r\!-\!2       \\
        & v\!+\!r        &                \\
 \end{array}\quad
&
{\rm for }\ 1\LEQ v\LEQ r-3.
\end{array}
$
\end{center}
}
The permutations $g_v=w_{P^{(v)}}$, $0\LEQ v\LEQ r-2$, are given by
\par\vspace{-2ex}
\begin{center}
$\begin{array}{l}
g_{0}=[1,2,r+1,3,r+2,\ldots,r-1,2r-2,r]
,
\\[1ex]
g_{(r-2)}=[r-1,1,r,2,r+1,\ldots,r-2,2r-3,2r-2],\quad {\rm and\ for}\quad 1\LEQ v\LEQ r-3,\\[1ex]
g_v=[v\!+\!1,1,v\!+\!2,2,\ldots,v,2v\!+\!1,  2v%
\!+\!2,v\!+\!r\!+\!1,\ldots,v\!+\!r\!-\!1,2r\!-\!2,v\!+\!r].
\end{array}$
\end{center}
\par\vspace{-2ex}
Since $P^{(v)}$ is admissible, $g_v\in Z(\lambda)$, $0\LEQ v\LEQ r-2$,
and it is immediate from a consideration of the tableaux $t_v$,
that $g_v$ is not a prefix of $g_{v'}$ if $v\neq v'$.
\end{example}
\begin{theorem}
\label{prop:3.15a}
Let $r\GEQ 3$, $n=2r-2$ and let $\lambda$ be the composition
$(1,2^{r-2},1)$ of $n$.
Then $\mathcal E^{(\lambda)}=\{P^{(v)}\colon 0\LEQ v\LEQ r-2\}$, where $P^{(v)}$ is as
described in Example~\ref{ex:3.14a}, and $\mathcal E_s^{(\lambda)}=\{P^{(0)},P^{(r-2)}\}$.
So, $|\mathcal E_s^{(\lambda)}|=2$ and $|\mathcal E^{(\lambda)}|=r-1$.
\end{theorem}
\begin{proof}
Let $y\in Y(\lambda)$ and let $D=D(y,\lambda)$ so that $y=w_D$.
Then $D=\{(1,j_{1})\}\cup
\{(i,j)\colon 2\LEQ i\LEQ r-1,\ j=j_i,j_i'\}\cup\{(r,j_r')\}$,
where we write $j_i<j_i'$ if $2\LEQ i\LEQ r-1$.
Since $\lambda'=(r,r-2)$ and $D$ is admissible, it has a path $\pi$ of
length $r$ and a 2-path $\Pi=(\pi_1,\pi_2)$ containing all its nodes.
By Theorem~\ref{thm:2.14a}, we may assume that this 2-path is
ordered.
Then $(i,j_i)\in\pi_1$ and $(i,j_i')\in\pi_2$ for $2\LEQ i\LEQ r-1$.
So $j_{2}\LEQ\cdots\LEQ j_{r-1}$ and $j_{2}'\LEQ\cdots\LEQ j_{r-1}'$.
Since $\pi$ has a node on each row, $(1,j_{1}),(r,j_r')\in\pi$,
$j_{1}\LEQ j_r'$, $j_{1}\LEQ j_{2}'$, and $j_{r-1}\LEQ j_r'$.
\par
If $j_{1}\LEQ j_{2}$, we may assume that
$(1,j_{1}),(r,j_r')$ belong to $\pi_1$.
Then $\Pi=(\pi_1,\pi_2)$ is ordered and has type $\lambda'=(r,r-2)$.
Since $D=s(\Pi)$ and $D(\Pi)=P^{(0)}$, invoking Proposition~\ref{propOrdSpecial}(ii) we get $D=P^{(0)}$.

\par
If $j_{r-1}'\LEQ j_r'$, we may assume that
$(1,j_{1}),(r,j_r')$ belong to $\pi_2$.
Then $\Pi=(\pi_1,\pi_2)$ is ordered and has type $\lambda'$.
Since $D=s(\Pi)$ and $D(\Pi)=P^{(r-2)}$, Proposition~\ref{propOrdSpecial}(ii) now ensures that $D=P^{(r-2)}$.

Now suppose that $j_{2}<j_{1}$ and $j_r'<j_{r-1}'$.
This cannot occur for $r=3$ since in this case the node $(2,j_{2})$
would be to the left of $(1,j_{1})$ and the node $(2,j_{2}')$ would
be to the right of $(3,j_{3}')$ contradicting the fact that $D$ has a path of length 3.
Hence, $r\GEQ4$ and $(2,j_{2}'),(r-1,j_{r-1})$ belong to $\pi$.
So, for some $v$, with $2\LEQ v\LEQ r-2$, we have $j_v'\LEQ j_{v+1}$.

Let $\rho_1$, $\rho_2$, $\rho_3$ be the paths in $D$ with
$s(\rho_1)=\{(2,j_2),\, (3,j_3),\,\ldots, (v,j_v)\}$,
$s(\rho_3)=\{(v+1,j_{v+1}'),\, (v+2,j_{v+2}'),\ldots, (r-1,j_{r-1}')\}$ and
$s(\rho_2)=D-(s(\rho_1)\cup s(\rho_3))$.
Note that the condition $j_v'\LEQ j_{v+1}$ ensures that $\rho_2$ is indeed a path.
It follows that $D=s(\tilde\Pi)$, where $\tilde\Pi=(\rho_1,\rho_2,\rho_3)$ is an ordered 3-path in $D$ with $D(\Pi)=P^{(v-1)}$.
By Proposition~\ref{propOrdSpecial}(i), $D=P^{(v-1)}$ in view of the fact that $P^{(v-1)}$ is admissible and $P^{(v-1)}=D(w_{P^{(v-1)}},\lambda)$.
(In particular, this shows that $j_v'\LEQ j_{v+1}$ for precisely one $v$ with $0\LEQ v\LEQ r-2$.)
\par
This establishes that
$\mathcal E^{(\lambda)}\subseteq\{P^{(v)}\colon 0\LEQ v\LEQ r-2\}$.
However, since each $g_v\in Z(\lambda)$ and no $g_v$ is a prefix
of any other one, each $g_v\in Y(\lambda)$.
So, $\mathcal E^{(\lambda)}=\{P^{(v)}\colon 0\LEQ v\LEQ r-2\}$.
\end{proof}
\begin{corollary}
\label{cor:3.16a}
Let $r\GEQ 3$, $s\GEQ1$, $t\GEQ1$, $n=2r+s+t-4$ and let $\lambda$ be the
composition $(1^s,2^{r-2},1^t)$ of $n$.
Then $|\mathcal E_s^{(\lambda)}|=2$ and $|\mathcal E^{(\lambda)}|=r-1$.
\end{corollary}
\begin{proof}
Apply Proposition~\ref{prop:3.2a} $t-1$ times to the case of the composition
$(1,2^{r-2},1)$,
then apply Remark~\ref{res:10a},
then apply Proposition~\ref{prop:3.2a} $s-1$ times,
and finally apply Remark~\ref{res:10a}.
The diagrams in $\mathcal E^{(\lambda)}$ are obtained
from the diagrams in Example~\ref{ex:3.14a} by extending the long
column of nodes upward by $s-1$ nodes and downward by $t-1$ nodes.
\end{proof}
\end{section}
%

\providecommand{\bysame}{\leavevmode\hbox to3em{\hrulefill}\thinspace}
\providecommand{\MR}{\relax\ifhmode\unskip\space\fi MR }
\providecommand{\MRhref}[2]{%
  \href{http://www.ams.org/mathscinet-getitem?mr=#1}{#2}
}
\providecommand{\href}[2]{#2}

\end{document}